\documentclass[11pt]{article}

\usepackage{graphicx, color}

\usepackage{amssymb}

\usepackage{graphicx}
\usepackage{pgf,tikz}
\usetikzlibrary{arrows}

\usepackage{amsthm}

\usepackage{epstopdf}

\usepackage{amsmath}
\usepackage{amsfonts}

\def\phi{{\varphi}}

\newcommand{\CO}[2]{ \left\langle #1 , #2 \right\rangle}

\DeclareSymbolFont{AMSb}{U}{msb}{m}{n}
\DeclareMathSymbol{\N}{\mathbin}{AMSb}{"4E}
\DeclareMathSymbol{\Z}{\mathbin}{AMSb}{"5A}
\DeclareMathSymbol{\R}{\mathbin}{AMSb}{"52}
\DeclareMathSymbol{\Q}{\mathbin}{AMSb}{"51}
\DeclareMathSymbol{\I}{\mathbin}{AMSb}{"49}
\DeclareMathSymbol{\C}{\mathbin}{AMSb}{"43}

\def\be{\begin{equation}}
\def\ee{\end{equation}}
\def\ber{\begin{eqnarray}}
\def\eer{\end{eqnarray}}

\def\beq{\begin{equation}}
\def\eeq{\end{equation}}

\def\Z{{\mathbb{Z}}}

\def\IR{{\mathbb{R}}}

\newcommand{\E}[0]{ \varepsilon}

\newcommand{ \pOm}{\partial \Omega}

\begin{document}

\addtolength{\textheight}{0 cm} \addtolength{\hoffset}{0 cm}
\addtolength{\textwidth}{0 cm} \addtolength{\voffset}{0 cm}

\newenvironment{acknowledgement}{\noindent\textbf{Acknowledgement.}\em}{}

\setcounter{secnumdepth}{5}

 \newtheorem{proposition}{Proposition}[section]
\newtheorem{theorem}{Theorem}[section]
\newtheorem{lemma}[theorem]{Lemma}
\newtheorem{coro}[theorem]{Corollary}
\newtheorem{remark}[theorem]{Remark}
\newtheorem{extt}[theorem]{Example}
\newtheorem{claim}[theorem]{Claim}
\newtheorem{conj}[theorem]{Conjecture}
\newtheorem{definition}[theorem]{Definition}
\newtheorem{application}{Application}
\newtheorem{exam}{Example}[section]
\newtheorem{thm}{Theorem}[section]
\newtheorem{prop}{Proposition}[section]

\newtheorem*{assumption}{Assumptions on $a(x)$}

\newtheorem*{thm*}{Theorem A}

\newtheorem{corollary}[theorem]{Corollary}

\title{Supercritical elliptic  problems on nonradial  domains via a  nonsmooth variational approach\footnote{Both authors  are pleased to acknowledge the support of the  National Sciences and Engineering Research Council of Canada.}}
\author{Craig Cowan\footnote{University of Manitoba, Winnipeg Manitoba, Canada, craig.cowan@umanitoba.ca} \quad  Abbas Moameni \footnote{School of Mathematics and Statistics,
Carleton University,
Ottawa, Ontario, Canada,
momeni@math.carleton.ca} }

\date{}

\maketitle

\vspace{3mm}

\begin{abstract}  In this paper we are interested in positive classical solutions of 
\begin{equation} \label{eqx}
\left\{\begin{array}{ll}
-\Delta u = a(x) u^{p-1} &  \mbox{ in } \Omega, \\
u>0    &   \mbox{ in } \Omega, \\
u= 0 &   \mbox{ on } \pOm,
\end {array}\right.
\end{equation}
where $\Omega$ is a bounded annular domain (not necessarily an annulus) in  $\IR^N$ $(N \ge3)$ 
  and  $ a(x)$ is a nonnegative continuous function.
We show the existence of a classical positive solution for a range of supercritical values of $p$ when the problem enjoys certain mild  symmetry and monotonicity conditions. 
As a consequence of our results,  we shall show that (\ref{eqx}) has $\Bigl\lfloor\frac{N}{2} \Bigr\rfloor$ (the  floor of $\frac{N}{2}$)   positive nonradial solutions when $ a(x)=1$ and  $\Omega$ is an annulus with certain assumptions on the radii.  We also obtain  the existence of positive solutions  in the case of toroidal domains.
Our approach is  based on a new  variational principle that allows one to deal with supercritical problems variationally by limiting the corresponding functional on a proper convex subset instead of the whole space at the expense of a mild invariance property.

\end{abstract}

\noindent
{\it \footnotesize 2010 Mathematics Subject Classification: 35J15, 35A15, 35A16, 35B07  {\scriptsize }  }	 \\
{\it \footnotesize Key words:    Supercritical elliptic equations, Variational and topological methods. } {\scriptsize }

\section{Introduction}


  In this paper we are interested in positive classical solutions of 
\begin{equation} \label{eq}
\left\{\begin{array}{ll}
-\Delta u = a(x) u^{p-1} &  \mbox{ in } \Omega, \\
u>0    &   \mbox{ in } \Omega, \\
u= 0 &   \mbox{ on } \pOm,
\end {array}\right.
\end{equation}
where $\Omega$ is a bounded domain in $ \IR^N$ $(N \ge 3)$ and where $ p>1$.  We will be interested in both questions of existence and multiplicity of positive solutions.    The domains will have certain symmetry that allows us to consider supercritical values of $p$.  One main class of domains will be annular domains (not necessarily an annulus).   The function  $ a(x)$ is a nonnegative function with some symmetry assumptions, for instance, one  can take $a(x)=1.$

 For $N \ge 3$ and $a(x)=1$  the critical  exponent  
 $2^*:=\frac{2N}{N-2}$ plays a crucial role and for $ 2<p<2^*$ a variational approach shows the existence of a smooth positive solution of (\ref{eq}); a crucial step here is that $H_0^1(\Omega)$ is compactly imbedded in $L^{p}(\Omega)$, see for instance the book  \cite{STRUWE}. 
    For $ p \ge 2^*$ the well known Pohozaev identity \cite{poho} shows there are no positive solutions of (\ref{eq}) provided $ \Omega$ is star shaped.    For general domains in the critical/supercritical case, $ p \ge 2^*$, the existence versus nonexistence of positive solutions of (\ref{eq}) is an overly   challenging question. 
    
    In  the case of a radial annulus, $ \Omega=A:=\{x \in \IR^N: 0<R_1<|x|<R_2 < \infty \}$,  there is a positive radial solution, for any $p>2$, which relies on  the compact imbedding of $H^1_{0,rad}(A)$ (this is the space of $H_0^1(A)$ radial functions) into $L^{p}(A)$.  For domains that are slight perturbations of $A$ one can try and linearize around the radial solution.   In the purely critical problem $ p = 2^*$ it was shown in \cite{Bahri-Coron} that there is a solution provided the domain is not contractible.    However,   finding positive non-radial solutions for the supercrital case has been a very delicate issue as well; for instance we refer to the recent paper \cite{Weth_annulus}.  Before  providing a more comprehensive background on this problem, we shall begin with our contribution on the subject in this paper.

\subsection{Domains of double revolution}    Unless explicity stated we are always assuming our domains will be domains of double revolution.  Consider writing $ \IR^N=\IR^{m} \times \IR^{n} $ where $ m,n \ge 1$ and $m+n=N$.
We define the variables $s$ and $t$ by
\[ s:= \left\{ x_1^2 + \cdot \cdot \cdot  + x_{m}^2\right\}^\frac{1}{2}, \qquad t:=\left\{x_{m+1}^2 + \cdot \cdot \cdot + x_N^2 \right\}^\frac{1}{2}.  \] We say that $\Omega \subset \IR^N$ is a \emph{domain of double revolution} if it is invariant under rotations of the first $m$ variables and also under rotations of the last $n$ variables.  Equivalently, $\Omega$ is of the form $ \Omega=\{ x \in \IR^N: (s,t) \in U \}$ where $U$ is a domain in $ \IR^2$ symmetric with respect to the two coordinate axes.  In fact,
\[ U= \big \{ (s,t) \in \IR^2:  x=(x_1=s, x_2=0, ... , x_m=0,  x_{m+1}=t, ... , x_N =0 ) \in \Omega \big \},\] is the intersection of $\Omega$ with the $(x_1,x_{m+1})$ plane.  Note that $U$ is smooth if and only if $\Omega$ is smooth. 
We denote $\widehat{\Omega}$ to  be the intersection of $U$ with the first quadrant of $\IR^2$, that is,
\begin{equation}\label{omegahat}\widehat{\Omega}=\big\{(s,t) \in U: \,\, s> 0, \,\,  t>0 \big\}.\end{equation}
Using polar coordinates we can write $ s= r \cos(\theta),$ $ t = r \sin(\theta)$ where $ r=|x|= |(s,t)|$ and $ \theta$ is the usual polar angle in the $ (s,t)$ plane. \\ 
\begin{figure}[htbp]
\centerline{\includegraphics[width=3in, height=3in]{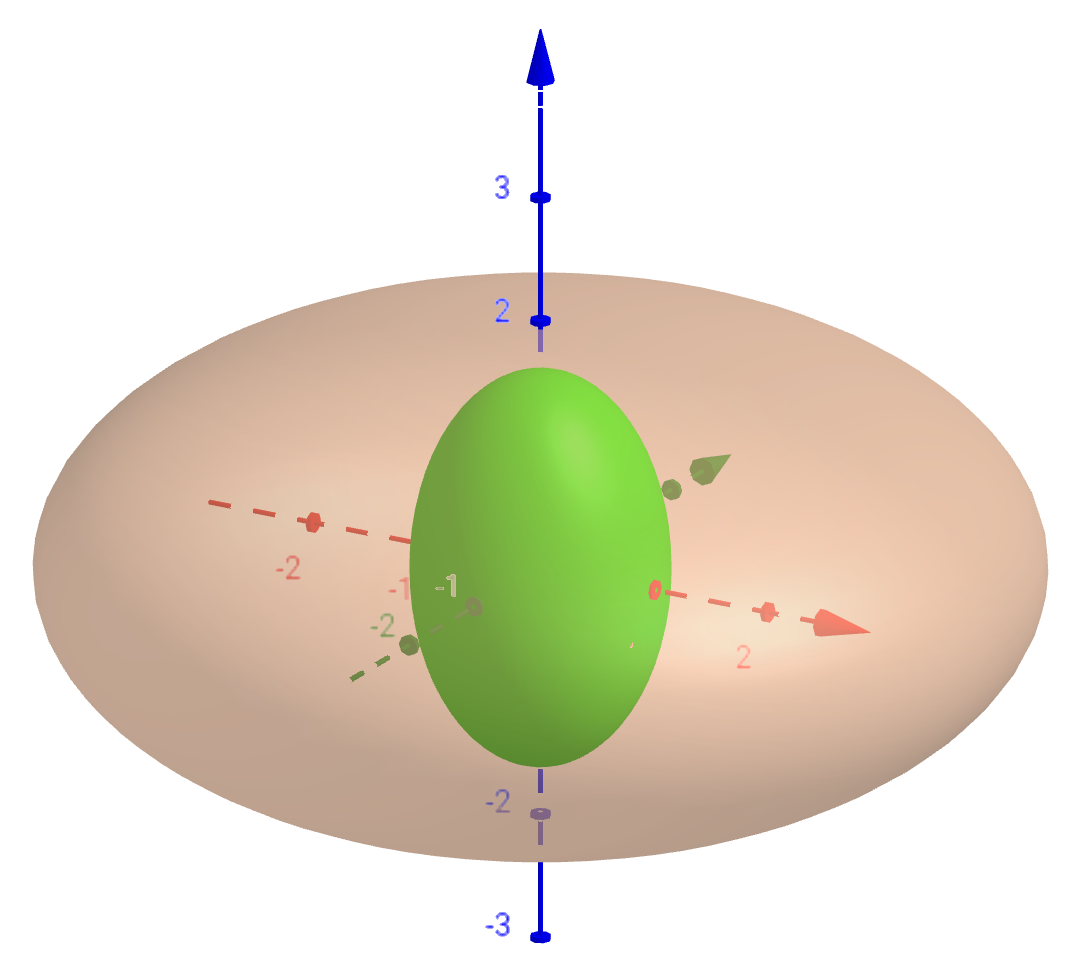}}
\caption{A solid bounded between the ellipsoids forms an annular domain with monotonicity.}
\label{fig}
\end{figure}
The domains under the consideration  will be annular and toroidal domains with a certain monotonicity (or convexity) assumption in the a polar angle.  All domains will be bounded domains in $ \IR^N$ with smooth boundary unless otherwise stated.   To describe the domains in terms of the above polar coordinates we will write \begin{equation}\label{omegahat+} \widetilde{\Omega}:=\big \{ (\theta,r): (s,t) \in \widehat{\Omega} \big \}.  
\end{equation}

We shall now discuss two important class of domains of double revolution that we are considering in this paper, namely,  annular domains and toroidal domains.\\

\noindent
\textbf{Annular domains.}  We begin by considering an explicit annular domain in $ \IR^N$ and then we will generalize. The first example would be an annulus centred at the origin with inner radius $ R_1$ and outer radius $ R_2$; $\Omega=\{x \in \IR^N:  R_1<|x|<R_2\}$.   Then we have $U=\{(s,t):  R_1^2 <s^2+t^2 <R_2^2\}$ and finally we have $ \widetilde{\Omega}=\{( \theta, r):   g_1(\theta)<r<g_2(\theta), \theta \in \left( 0 , \frac{\pi}{2} \right)\}$ where $ g_1(\theta)=R_1$ and $ g_2(\theta)=R_2$.    

We can now consider a  slightly more general version where the inner and outer boundaries are replaced with ellipsoids instead of balls.   Take $ \Omega$ to have outer boundary given by  the ellipsoid 
\[ \sum_{k=1}^{m}  \frac{x_k^2}{a^2}+ \sum_{k=m+1}^{N}  \frac{x_k^2}{b^2}  =1,\] and the inner boundary given by  \[\sum_{k=1}^{m}   \frac{x_k^2}{c^2}+ \sum_{k=m+1}^{N}  \frac{x_k^2}{d^2}  =1,\] where $ a,b,c,d>0$ are chosen such that the resulting domain is an annular region. 
Note in this case we have 
\[ \widehat{\Omega}=\left\{ (s,t): s,t>0, \; \frac{s^2}{a^2}+ \frac{t^2}{b^2}<1 \mbox{ and } \frac{s^2}{c^2}+ \frac{t^2}{d^2}>1 \right\},\]
and \begin{equation*} 
\widetilde{\Omega}=\left\{ (\theta,r): g_1(\theta)< r < g_2(\theta), \theta \in \left(0, \frac{\pi}{2}\right) \right\}, 
\end{equation*}
where the functions $g_1$  and $g_2$ are given by 
\[ g_2(\theta) = \frac{1}{ \left( \frac{1}{b^2} + \sin^2(\theta) \left( \frac{1}{a^2} - \frac{1}{b^2} \right) \right)^\frac{1}{2} } \mbox{ and }  g_1(\theta) = \frac{1}{ \left( \frac{1}{d^2} + \sin^2(\theta) \left( \frac{1}{c^2} - \frac{1}{d^2} \right) \right)^\frac{1}{2} }.\]   From this example we now introduce the idea of an \emph{annular domain with monotonicity}.     Consider the annular region in the $(s,t)$ variables if we make the further restriction   $ 0 <d \le c <a \le b$;  note we can consider this region as being obtained by starting with two concentric spheres in the $(s,t)$ plane and vertically compressing the outer sphere and vertically stretching the inner one and then $ U$ is the region between the two deformed spheres.   In terms of $g_i$ note that $ g_1$ is increasing on $(0, \frac{\pi}{2}) $ and $g_2$ is decreasing on $ (0, \frac{\pi}{2})$. \\

\begin{definition}  We refer to a domain of double revolution in $\R^{N}$ with $N=m+n$ an annular domain if its  associated domain $\widehat \Omega$ in the $(s,t)$ plane in $\R^2$ is of  the form  

\begin{equation} 
\widetilde{\Omega}=\left\{ (\theta,r): g_1(\theta)< r < g_2(\theta), \theta \in \left(0, \frac{\pi}{2}\right) \right\} 
\end{equation} 
 in polar coordinates.  Here  $ g_i>0$ is smooth on $ [0, \frac{\pi}{2}]$ with $ g_i'(0)=g_i'( \frac{\pi}{2})=0$ and $ g_2(\theta)> g_1(\theta)$ on $ [0, \frac{\pi}{2}]$.   We also call  $\Omega$ an  \emph{annular domain with monotonicity} if  $ g_1$ is increasing and $ g_2$ is decreasing on $ (0, \frac{\pi}{2})$. 
 
\end{definition}  
 See Figure 1 for  an example of an annular domain with monotonicity in $ \IR^3$, see Figure 2 for its representation $\widehat \Omega$ in the $(s,t)$ plane.  We also refer to Figure 3 for the representation $\widetilde \Omega$ of $\widehat \Omega$ in polar coordinates.

\begin{figure}[tbp]
  \centering
 \hspace{-1.5 cm}  \begin{minipage}[b]{0.4\textwidth}
    \includegraphics[width=3in, height=3in]{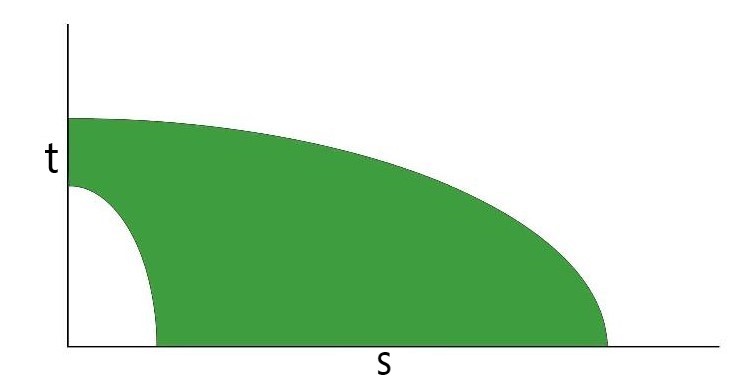}
    \caption{$\widehat \Omega$ in $s-t$ coordinates.}
  \end{minipage}
  \hspace{1.5 cm} 
  \begin{minipage}[b]{0.4\textwidth}
    \includegraphics[width=3in, height=3in]{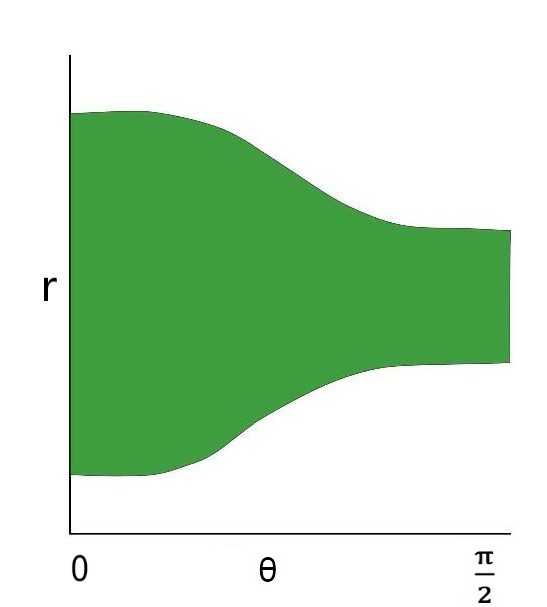}
    \caption{$\widetilde{\Omega}$ in polar polar coordinates}
  \end{minipage}
\end{figure}  

\vspace{.5 cm}

\noindent
\textbf{Toroidal domains.}  Consider a standard torus in $ \IR^3$  given by $ (\sqrt{x^2+y^2}-a)^2 + z^2<b^2$ where $ a>b>0$.   We now adapt this to handle more  double domains of revolution.  So we consider $ (s-a)^2+ t^2<b^2$ or in terms of $x$ we have 
\[ \left( \left\{ \sum_{k=1}^m x_k^2 \right\}^\frac{1}{2}-a \right)^2+ \sum_{k=m+1}^N x_k^2 <b^2.\] Now note $ \widehat{\Omega}$ is the upper half of the circle of radius $ b$ in the $(s,t)$ place centered at $ (a,0)$.     We now write this in terms of polar coordinates as before.   Let $ 0<\theta_0<\frac{\pi}{2}$ be such that $ \sin(\theta_0)= \frac{b}{a}$.   Then we can write the domain as  $ g_1(\theta)<r<g_2(\theta)$  for $ 0<\theta<\theta_0$ where  $ g_i$ is given by
\[ g_1(\theta)= a \cos(\theta) - \sqrt{ b^2-a^2 \sin^2(\theta)}, \quad g_2(\theta)=a \cos(\theta) + \sqrt{b^2-a^2 \sin^2(\theta)}.\]   A computation shows that $ g_1$ is increasing and $ g_2$ is decreasing on $ (0, \theta_0)$.   

We can consider a generalized torus where we take more general functions $g_i$.   We will refer to a  \emph{toroidal domain with monotonicity} to be a torus with general $g_i$  but with the added assumption that $ g_1$ is increasing and $ g_2$ is decreasing on $(0, \theta_0)$  where  $\theta_0$ is a unique point such that $g_1(\theta_0)=g_2(\theta_0)$.  In both cases we define 
\[\widetilde{\Omega}:=\big \{ (\theta,r): (s,t) \in \widehat{\Omega} \big\}=\big\{ (\theta,r): g_1(\theta)<r<g_2(\theta), \theta \in (0,\theta_0)\big\}.\]  \\

\noindent
\textbf{$(\mathcal{A})$: Conditions on $a(x)$.} Throughout the paper, we always assume that $a$ is a continuous   function of $(s,t)$ that is $a(x)=a(s,t).$  Moreover,   we say that $a$ satisfies \textbf{$(\mathcal{A})$} if  $a$ is  a continuously  differentiable function with respect to $(s,t)$  and  $sa_t-ta_s \leq 0$ in $ \widehat{\Omega}$. \\

We give some examples of $a(x)$ which satisfy $(\mathcal{A})$:
\begin{itemize}
    \item $a(x)=1$, 
    \item $ a(x)=|x|^\alpha$ where $ \alpha \in \R$, in this case the equation corresponds to the well known H\'enon equation.  Here $a(s,t)=(s^2+t^2)^{\frac{\alpha}{2}}.$
    \item $ a(x)=h\big (\sqrt{x_1^2 + \cdot \cdot \cdot  + x_{m}^2}\big)$ where $h$ is continuously differentiable and $h'\geq 0.$ In this case $a(s,t)=h(s).$ 
    
\end{itemize}

\subsection{Main results}

\begin{thm}\label{main} (Existence results)
\begin{enumerate} \item (Annular domains with monotonicity)
    Suppose $ \Omega \subset \R^N=\R^m \times \R^n$ is an annular domain with monotonicity in $ \IR^N$ with $ n \le m$ and 
    \begin{equation} \label{max_of}
    2<p<\frac{2(n+1)}{n-1}=\max \left\{ \frac{2(n+1)}{n-1},  \frac{2(m+1)}{m-1} \right\}, 
    \end{equation}
 or $ 2<p<\infty$ if  $n=1$ and we also suppose that $ a$ satisfies $(\mathcal{A})$.   Then there is a positive classical solution of (\ref{eq}).
 
 \item (Annular domains without monotonicity)   Suppose $ \Omega$ is an annular domain in $ \IR^N$ with 
    \begin{equation} 
    2<p<\min \left\{ \frac{2(n+1)}{n-1},  \frac{2(m+1)}{m-1} \right\}, 
    \end{equation}
 or $ 2<p<\infty$ if  $m= n=1.$  Then there is a positive classical solution of (\ref{eq}).   
 
 \item (Toroidal domains  with monotonicity)  Suppose $ \Omega$ is a toroidal domain with monotonicity in $ \IR^N$ with 
    \begin{equation} 
    2<p<\frac{2(n+1)}{n-1}, 
    \end{equation} and $ a$ satisfies $(\mathcal{A})$. 
     Then there is a positive classical solution of (\ref{eq}). 
 \end{enumerate} 
 
In all cases of domains with monotonicity the solution $u$ will have a certain monotonicity (this will be clear from the proof). 
\end{thm}

\begin{remark}  \label{remark_about_thm1}
  \begin{enumerate}  
  \item We believe the above result regarding an annular domain without monotonicity  is well known and is the  gain of compactness one expects with an annular domain.
  We are including it since we get the result as a by product of our approach and also so we can compare with the case where one adds monotonicity.   See Theorem  \ref{compact_no_mono} for the related imbedding with and without monotonicity.   It should be noted that the monotonicity is needed for both an improved imbedding to hold and also for the pointwise invariance property to hold, which is crucial to our approach. 
  
  \item For the toroidal domains we can handle both the case with and without monotonicity.  We are not stating a result for the case without monotonicity.  We point out that monotonicity here is not allowing an increased range of $p$, see Corollary \ref{torus_imbed}.
  
  \item The function $a(x)$ is not adding any compactness to the problem; the same proof works if  $ a(x)=1$. 
  
  \item In this work we are dealing with domains of double revolution.  We suspect similar arguments would work for \emph{domains of triple revolution} and higher order analogs.
  
   \end{enumerate} 
  \end{remark}

Next we discuss the case when $a(x)=a(|x|)$ is radial,  and $\Omega$ is an annulus,  
 that is $\Omega=\{x:\,  R_1< |x|<R_2\}:$ 
 \begin{equation} \label{eqz}
\left\{\begin{array}{ll}
-\Delta u = a(|x|) u^{p-1} &  \mbox{ in } \Omega, \\
u>0    &   \mbox{ in } \Omega, \\
u= 0 &   \mbox{ on } \pOm.
\end {array}\right.
\end{equation}
 
 Note that in this case the problem (\ref{eqz}) always has a positive radial solution.   We remark that any radial function $a(x)=a(|x|)$
 can be written in the $(s,t)$ coordinates as follows:
 \[{\bf a}(s,t)=a(|x|)=a(\sqrt{s^2+t^2}),\]
 for any $m,n\geq 1$ with $N=m+n.$
 In the following theorem we shall prove that one can also obtain non-radial solutions.
\begin{thm}\label{nnr} Take $m,n\geq 1$ with $N=m+n.$ Let $u=u(s,t)$ be the  solution of (\ref{eqz}) obtained in Theorem \ref{main} with
\[ s:= \left\{ x_1^2 + \cdot \cdot \cdot  + x_{m}^2\right\}^\frac{1}{2}, \qquad t:=\left\{x_{m+1}^2 + \cdot \cdot \cdot + x_N^2 \right\}^\frac{1}{2}.  \]Define 
\[\lambda_1=\inf_{0 \neq w \in H^1_0(\Omega)}\frac{\int_\Omega |\nabla w|^2 \,dx}{\int_\Omega \frac{| w|^2 }{|x|^2}\,dx}.\]

If  $p-2>\frac{2N}{\lambda_1}$ then $u$ is not radial. 
\end{thm} 

Note that $ \lambda_1$ is nothing more then the optimal constant in the classical Hardy inequality on $ \Omega$.  Since $\Omega$ does not contain the origin and is not an exterior domain we see that $ \lambda_1$ is attained and hence $ \lambda_1> \frac{(N-2)^2}{4}$, which is the well known optimal constant in the classical Hardy inequality.    We would also like to remark that the inequality $p-2>\frac{2N}{\lambda_1}$  provides a  relatively simple condition  for the symmetry breaking in (\ref{eqz}).\\
In the next theorem we address the multiplicity of positive solutions for (\ref{eqz}).
\begin{thm} \label{multiplicity_100} For each $1\leq k\leq \Bigl\lfloor\frac{N}{2} \Bigr\rfloor$ (where   $ \lfloor x \rfloor$ is the floor of $x$)  
the equation (\ref{eqz}) has $k$ positive distinct nonradial solutions if  $2+\frac{2N}{ \lambda_1}<p< \frac{2 k+2}{k-1} $ for $k>1$, and    $2+\frac{2N}{ \lambda_1}<p$ for $k=1.$
\end{thm} 

Note the above gives various results.  For instance lets suppose $N$ is even and we are interested under what assumptions on the parameters do we have $  \frac{N}{2}$ solutions.    It would be sufficient that 
\[ 2 + \frac{2N}{\lambda_1} <p< \frac{2(N+2)}{N-2},\] hold.   Under suitable assumptions on $ R_1,R_2$ we can show $ \lambda_1$ gets large and hence we see there is a range of $p$ where this holds.  Indeed,
we use the result of Theorem \ref{multiplicity_100} to prove the following Corollary.

\begin{coro}\label{annul_resu}  The following assertions hold;
\begin{enumerate}
\item For $ 0<R_1<R_2<\infty$ and sufficiently large $ p$ there is a nonradial solution of (\ref{eqz}). 
\item For fixed $2<p<\frac{2\Bigl\lfloor\frac{N}{2} \Bigr\rfloor+2}{\Bigl\lfloor\frac{N}{2} \Bigr\rfloor-1}$ and sufficiently large $\lambda_1$ there are $\Bigl\lfloor\frac{N}{2} \Bigr\rfloor$ distinct  positive   nonradial solutions of (\ref{eqz}). For instance,  under either of the following conditions $\lambda_1$ can be sufficiently large and therefore there are $\Bigl\lfloor\frac{N}{2} \Bigr\rfloor$ distinct  positive   nonradial solutions of (\ref{eqz}):
\begin{itemize}
\item[2-a:] Let  $ R_1=R$ and $ R_2=R+1.$  Then $\lambda_1$  is sufficiently large for large values of  $R$.   Note by scaling we can take $ R_1=1$ and $ R_2=1+\frac{1}{R}$ and get the same result for large $R$. 

\item[2-b:] Let $ R<\gamma(R)$ with $ \frac{\gamma(R)}{R} \rightarrow 1$ as $ R \rightarrow \infty$.   With $ \Omega_R=\{x \in \IR^N: R<|x|<\gamma(R) \}$ then for large enough $R$  the  $\lambda_1$  corresponding to $ \Omega_R$ is sufficiently large.

\end{itemize}
\end{enumerate} 
\end{coro}

Note that {\it {2}-a}  is a particular case of {\it {2}-b}. We are stating them as two separate  assertions as {\it {2}-a} is widely studied  in the literature while {\it {2}-b}  is a less known generalization of{ \it {2}-a}.\\
We also remark that the results in Corollary \ref{annul_resu} recovers and complements some of the existing results on the subject (see \cite{Ann2, Weth_annulus} and the references therein).




\subsection{The approach} 

Our plan is to prove existence for (\ref{eq}) by making use of  a new abstract variational principle established recently in \cite{Mo3} (see also \cite{Mo2, Mo4, Mo} for more applications).  To be more specific,
let $V$ be a reflexive Banach space, $V^*$ its topological dual  and let $K$ be a non-empty convex and weakly closed subset of $V$.
Assume that  $\Psi : V \rightarrow \mathbb{R} \cup \{+\infty\}$ is a proper, convex, lower semi-continuous function which is  G\^ateaux differentiable on K. The  G\^ateaux derivative of $\Psi$ at  each point $u \in K$ will be denoted by  $D\Psi(u)$. The restriction of $\Psi$ to $K$ is denoted by $\Psi_K$ and defined by
\begin{eqnarray}
\Psi_K(u)=\left\{
  \begin{array}{ll}
      \Psi(u), & u \in K, \\
    +\infty, & u \not \in K.
  \end{array}
\right.
\end{eqnarray}
For a given functional $\Phi \in C^{1}(V, \mathbb{R})$ denote by $D \Phi \in V^*$ its derivative and consider the functional $I_K: V \to (-\infty, +\infty]$ defined by
 \begin{eqnarray*}
 I_K(u):= \Psi_K(u)-\Phi(u).
 \end{eqnarray*}
According to Szulkin \cite{szulkin}, we have the following definition for critical points of $I_K$.

\begin{definition}\label{ddd}
A point $u_0\in  V$ is said to be a critical point of $I_K$ if $I_K(u_0) \in \mathbb{R}$ and if it satisfies the following inequality
\begin{equation}
 \CO{D \Phi(u_0)}{ u_0-v} + \Psi_K(v)- \Psi_K(u_0) \geq 0, \qquad \forall v\in V,
\end{equation}
where $ \CO{ \cdot }{ \cdot } $ is the duality pairing between $V$ and its dual $V^*.$
\end{definition}

We also recall the notion of point-wise invariance condition from \cite{Mo3}.
\begin{definition}\label{def-invar}
We say that the triple $(\Psi, K, \Phi)$ satisfies the point-wise invariance condition at a 
 point $u_{0}\in V$ if there exists a convex G\^ateaux differentiable function $G: V\to \R$ and a point $v_0 \in K$ such that 
 \[D\Psi(v_0)+DG(v_0)=D \Phi(u_0)+DG(u_0).\]
\end{definition}

We shall now recall the following variational principle recently
established  in \cite{Mo3} (see also \cite{Mo}).
\begin{thm} \cite{Mo3} \label{vp}
Let $V$ be a reflexive Banach space and $K$ be a convex and
weakly closed subset of $V$. Let $\Psi : V \rightarrow
\mathbb{R}\cup \{+\infty\}$ be a convex, lower semi-continuous
function which is G\^ateaux differentiable on $K$ and let $\Phi
\in C^{1}(V, \mathbb{R})$. Assume that  the following two assertions hold:\begin{enumerate}
\item[$(i)$] The functional $I_K: V \rightarrow \mathbb{R}
\cup \{+\infty\}$ defined by $I_K(u):= \Psi_K(u)-\Phi(u)$ has a
critical point $u_{0}\in V$ as in Definition \ref{ddd}, and;
\\
\item[$(ii)$] the triple $(\Psi, K, \Phi)$ satisfies the point-wise invariance condition at the  
 point $u_{0}$.
\end{enumerate} 
Then $u_{0}\in K$ is a solution of the equation 
 \begin{equation} \label{equ1} D \Psi(u) =D \Phi(u).
 \end{equation}
 
 \end{thm}

For the convenience of the reader, by choosing appropriate  functions $\Psi, \Phi$ and a  convex set $K$ corresponding to our problem (\ref{eq}), we  shall provide a proof to a particular case of Theorem \ref{vp} applicable to this problem. \\
    We observe  that by picking $K$ appropriately one can gain compactness;  note the smaller we pick $K$ the more manageable  $I_K$ becomes which makes proving the existence of critical points of $I_K$ easier.   But this needs to be balanced with   the second part of the Theorem \ref{vp} where we need to verify the   point-wise invariance condition.  We will describe the $K$ we will pick in detail later but we mention that $K$ will be a collection of functions with some monotonicity in $\theta$ where $ \theta$ is the polar angle when writing $ (s,t)$ in polar coordinates.  This monotonicity in $\theta$ is exactly what allows us to obtain added compactness in problems on an annulus with monotonicity (as mentioned before there is added compactness from the symmetry of the domain and this is well known;  this added compactness from monotonicity is some additional compactness).

\subsection{Background}

Here we give some background on the the problem and for this we take $ a(x)=1$ and hence we consider 
\begin{equation} \label{eq_a=1}
\left\{\begin{array}{ll}
-\Delta u =  u^{p-1} &  \mbox{ in } \Omega, \\
u>0    &   \mbox{ in } \Omega, \\
u= 0 &   \mbox{ on } \pOm.
\end {array}\right.
\end{equation}
We assume $ \Omega$ a bounded smooth domain in $ \IR^N$.  When $N=2$ there is a positive smooth solution of (\ref{eq_a=1}) for any $ p>2$.  For $N \ge 3$ the critical  exponent   $2^*:=\frac{2N}{N-2}$ plays a crucial role and for $ 2<p<2^*$ a variational approach shows the existence of a smooth positive solution of (\ref{eq_a=1}).
       For general domains in the critical/supercritical case, $ p \ge 2^*$, the existence versus nonexistence of positive solutions of (\ref{eq_a=1}) presents a great degree of difficulties; see \cite{Bahri-Coron, Coron,del_p,man_1,man_2,man_3,wei,  A2, A3,Passaseo, Passaseo_2,shaaf, schmitt}.  Many of these results are very technical and some require  perturbation arguments. 
    
    The approach we will take is coming from some recent approaches in Neumann versions of 
    (\ref{eq}) given by 
    \begin{equation} \label{eq_Nball}
\left\{\begin{array}{rr}
-\Delta u + u = a(r) u^{p-1} &  \mbox{ in } B_1, \\
\partial_\nu u= 0 &   \mbox{ on } \partial B_1,
\end {array}\right.
\end{equation}
    where $B_1$ is the unit ball centered at the origin in $\IR^N$.  The interest here is in obtaining nontrivial solutions for values of $p>\frac{2N}{N-2}$. 
 In  \cite{first_rad_neum} they considered the variant of (\ref{eq_Nball}) given by $ -\Delta u + u = |x|^\alpha u^{p-1}$ in $B_1$ with $ \frac{ \partial u}{ \partial \nu}=0$ on $ \partial B_1$ (for Dirichelt versions of the H\'enon equation see, for instance, \cite{Ni,glad,cowan}).
 They proved the existence of a positive radial solutions of this equation with arbitrary growth using a shooting argument.   The solution turns out to be an increasing function.
They also perform numerical computations to see the existence of positive oscillating solutions.  In  \cite{serra_tilli} they considered  (\ref{eq_Nball})   along with the classical energy associated with the equation given by
  \[E(u):=\int_{B_1} \frac{ | \nabla u|^2 +u^2}{2}\, dx  -\int_{B_1} a(|x|)F(u)\,  dx,\] where $F'(u)=f(u)$ (they considered a more general nonlinearity).    Their goal was  to find critical points of $E$ over $H^1_{rad}(B_1):=\{ u \in H^1(B_1):  u \mbox{ is radial} \}$.  Of course since $f$ is supercritical the standard approach of finding critical points will present difficulties and hence their idea was to find critical points of $E$ over the cone $ \{ u \in H_{rad}^1(B_1): 0 \le u, \mbox{ $u$ increasing} \}$.   Doing this is somewhat standard but now the issue is the critical points don't necessarily correspond to critical points over $H_{rad}^1(B_1)$ and hence one can't conclude the critical points solve the equation.   The majority of their work was to show that in fact the critical points of $E$ on the cone are really critical points over the full space.   We mention that this work generated a lot of interest in this equation and many authors investigated these idea's of using monotonicity to overcome a lack of compactness.  For further results  regarding these Neumann problems on radial domains (some using these monotonicity ideas and some using other new methods) see  \cite{first_grossi, Weth, secchi, grossi_new, den_serra, add1,add2,add3,add4,add5,add6}.

 In \cite{ACL} we considered (\ref{eq_Nball}) using a new variational principle (see Theorem \ref{vp}).  We obtained positive solutions of (\ref{eq}); assuming the same assumptions as the earlier works.   In the case of $a(x)=1$ we obtained the existence of a positive nonconstant solution of (\ref{eq}), which was a known result, but our approach allowed us to deal directly with the   supercritical nonlinearity without the need for any cut off procedures.

In   \cite{cowan-abbas}  we examined the Neumann problem given by 
\begin{equation}   \label{eq_trans}
\left\{\begin{array}{ll}
-\Delta u + u= a(x) f(u), &  \mbox{ in } \Omega, \\
u>0,    &   \mbox{ in } \Omega, \\
\frac{\partial u}{\partial \nu}= 0, &   \mbox{ on } \pOm,
\end {array}\right.
\end{equation}
where $\Omega$ is a bounded domain in  $\IR^N$ which was a \emph{domain of $m$ revolution} with certain symmetry and where $a$ also satisfied some symmetry assumptions. 
  For the sake of the background one can take $ f(u)=u^{p-1}$ and hence a supercritical result would be if $ p>2*:=\frac{2N}{N-2}$.   In this case  we obtained positive nontrivial monotonic solutions of (\ref{eq_trans}) provided $ 2<p< 2_m^*:= \frac{2m}{m-2}$. For Neumann problems on general domains see \cite{100, 101, 102,104,103,105,50,106}.

We now return to the Dirichlet problems.   There have been many supercritical works that deal with domains that have certain symmetry, for instance, see \cite{clapp1,clapp2,clapp3, clapp4,clapp5,clapp6,clapp7}.  

In   the case of  the annulur domains the authors in  \cite{Ann100, Ann101,Ann102} examined subcritical or slightly supercritical problems on expanding annuli and obtained nonradial solutions.  In \cite{Ann2}  they obtain nonradial solutions to supercritical problems on expanding annulur domains.  In \cite{Ann1} they consider nonradial expanding annulur domains and they obtain the existence of positive solutions.   In \cite{wei, clapp5}  they consider domains with a small hole and obtain positive solutions.

We now consider the very recent work 
\cite{Weth_annulus} where they examined 
\begin{equation}   \label{recent_annulus}
\left\{\begin{array}{ll}
-\Delta u + u= a(x) |u|^{p-2} u, &  \mbox{ in } A, \\
u>0,    &   \mbox{ in } A, \\
u= 0, &   \mbox{ on } \partial A,
\end {array}\right.
\end{equation}
where $ A:=\{x \in \IR^N:  0<R_1<|x|<R_2<\infty\}$ and $ a(x)$ is positive, even with respect to $x_N$, axial symmetric with respect to the $x_N$ axis and, using the earlier defined coordinates,  $ a=a(r,\theta)$ satisfies $ a_\theta \le 0$ in upper half of the $A$. 
They work on a convex cone $\mathcal{K}$ which is characterized by monotonicity properties of the functions; the functions are increasing in $\theta$ (this idea of monotonicity has been used a lot in Neumann problems but is new for Dirichlet problems).  They consider the standard energy functional associated with (\ref{recent_annulus}) and they work on $\mathcal{K}$ and they also consider $\mathcal{N}_{\mathcal{K}}:=\{ u \in \mathcal{K}: I'(u)u=0 \}$ which is a Nehari type set adapted to $\mathcal{K}$.  Then then develop a mountain pass type argument that utilizes some technical aspects of an associated flow and then they use some involved arguments from dynamical systems to prove the existence of a solution.  They obtain a positive solution for all $p>2$. In the case of $a(x)=1$ they obtain results regarding nonradial solutions under certain assumptions on the radii of the annulus or the value of $p$.

\section{Preliminaries}

In this section we recall some important definitions and results from convex analysis and  minimax principles for lower semi-continuous functions. 

Let $V$ be a  real Banach  space and $V^*$ its topological dual  and let $\langle \cdot, \cdot \rangle $ be the pairing between $V$ and $V^*.$ 
The weak topology on $V$ induced by $\langle \cdot,\cdot  \rangle $ is denoted by $\sigma(V,V^*).$  A function $\Psi : V \rightarrow \mathbb{R}$ is said to be weakly lower semi-continuous if
\[\Psi(u) \leq \liminf_{n\rightarrow \infty} \Psi(u_n),\]
for each $u \in V$ and any sequence ${u_n} $ approaching $u$ in the weak topology $\sigma(V,V^*).$
Let $\Psi : V \rightarrow \mathbb{R}\cup \{+\infty\}$ be a proper convex  function. The subdifferential $\partial \Psi $ of $\Psi$
is defined  to be the following set-valued operator: if $u \in Dom (\Psi)=\{v \in V; \, \Psi(v)< +\infty\},$ set
\[\partial \Psi (u)=\{u^* \in V^*; \langle u^*, v-u \rangle + \Psi(u) \leq \Psi(v) \text{  for all  } v \in V\}\]
and if $u \not \in Dom (\Psi),$ set $\partial \Psi (u)=\varnothing.$ If $\Psi$ is G\^ateaux differentiable at $u,$ denote by $D \Psi(u)$ the derivative of $\Psi$ at $u.$ In this case  $\partial \Psi (u)=\{ D  \Psi(u)\}.$\\

We shall now recall some notations and results for the minimax principles for lower semi-continuous functions.
\begin{definition}\label{cp}
Let $V$ be a real Banach space,  $\Phi\in C^1(V,\mathbb{R})$ and $\Psi: V\rightarrow (-\infty, +\infty]$ be proper (i.e. $Dom(\Psi)\neq \emptyset$), convex and lower semi-continuous. 
A point $u\in  V$ is said to be a critical point of \begin{equation} \label{form}I:=  \Psi-\Phi \end{equation} if $u\in Dom(\Psi)$  and if it satisfies
the inequality
\begin{equation}
 \CO{D \Phi(u)}{ u-v}+ \Psi(v)- \Psi(u) \geq 0, \qquad \forall v\in V.
\end{equation}
\end{definition}

\begin{definition}\label{psc}
We say that $I$ satisfies the Palais-Smale compactness  condition (PS)   if
every sequence $\{u_n\}$ such that
\begin{itemize}\label{2}
\item  $I[u_n]\rightarrow c\in \mathbb{R},$
\item  $\CO{D \Phi(u_n)}{ u_n-v}+ \Psi(v)- \Psi(u_n) \geq -\E_n\|v- u_n\|, \qquad \forall v\in V,$
\end{itemize}
where $\E_n \rightarrow 0$, then $\{u_n\}$ possesses a convergent subsequence.
\end{definition}

The following theorem is  due to  A. Szulkin  \cite{szulkin}. 
\begin{thm}   \label{MPT} \cite{szulkin}
(Mountain Pass Theorem).  Suppose that
$I : V \rightarrow (-\infty, +\infty ]$ is of the form (\ref{form}) and satisfies the Palais-Smale   condition and  the Moutaint Pass Geometry (MPG):
\begin{enumerate}
\item $I(0)= 0$.
\item  There exists $e\in V$ such that $I(e)\leq 0$.
\item There exists some $\rho$ such that $0<\rho<\|e\|$ and for every $u\in V$ with $\|u\|= \rho$ one has $I(u)>0$.
\end{enumerate}
Then $I$ has a critical value $c>0$ which is  characterized by
$$c= \inf_{\gamma\in \Gamma}  \sup_{t\in [0,1]} I[\gamma(t)],$$
where   $\Gamma= \{\gamma\in C([0,1],V): \gamma(0)=0,\gamma(1)= e\}.$
\end{thm}

 \section{A variational approach towards the super-critical problem}

In this section, we first prove an adapted version of Theorem \ref{vp} applicable specifically to our problem and then we proceed with  the proof of our main results. For the first subsection our energy functionals are explicit but we leave $K$ somewhat arbitrary,  which could allow one to consider some other related problems.\\  
 \subsection{The elliptic problem on a general $K$} 
Consider the Banach space  $V= H_0^1(\Omega)\cap L^p(\Omega)$  equipped  with the following norm
\[\|u\|_V= \|u\|_{H_0^1(\Omega)}+\|u\|_{L^{p}(\Omega)},\]
and 
note that the duality pairing between $V$ and its dual $V^*$ is defined by
\[\langle u, u^*\rangle=\int_\Omega u(x) u^*(x)\, dx,\qquad \forall u\in V, \, \forall u^* \in V^*.\]
Our plan is 
 to apply Theorem \ref{vp} to  the Euler- Lagrange functional corresponding to problem (\ref{eq}), i.e., 

$$E(u):= \frac{1}{2}\int_{\Omega} |\nabla u|^2 dx- \frac{1}{p}\int_{\Omega}  a(x)|u|^p dx,$$
over a closed convex  set $K.$
To adapt Theorem \ref{vp} to our case, define $\Psi: V \to \R$ and $\Phi: V \to \R$ by 
\[\Psi(u)=\frac{1}{2}\int_{\Omega} |\nabla u|^2 dx,\]
and 
\[\Phi(u)=  \frac{1}{p}\int_{\Omega}  a(x)|u|^p dx. \]

We remark that even though $\Phi$ is not even well-defined on  $H_0^1(\Omega)$ for large $p$, but it is continuously differentiable on the space $V= H_0^1(\Omega)\cap L^p(\Omega)$.
Finally, let us introduce the  functional $E_K(u): V\rightarrow (-\infty, +\infty]$ defined by
\begin{equation}\label{E}
E_K(u):= \Psi_K(u)- \Phi(u)
\end{equation}
where
\begin{eqnarray}
\Psi_K(u)=\left\{
  \begin{array}{ll}
      \Psi(u), & u \in K, \\
    +\infty, & u \not \in K.
  \end{array}
\right.
\end{eqnarray}
Note that $E_K$  is indeed the Euler-Lagrange functional corresponding to (\ref{eq}) restricted to $K$. Here is an adapted  version of Theorem \ref{vp} applicable to our case.

\begin{prop}\label{prn}
Let   $V= H_0^1(\Omega)\cap  L^p(\Omega)$ and let $K$ be a  convex and closed subset of $V$. Suppose the following two assertions hold:
\begin{enumerate}
\item[$(i)$] The functional $E_K: V \rightarrow \mathbb{R} \cup \{+\infty\}$ defined in $(\ref{E})$  has a critical point $\bar u\in V$ as in Definition \ref{cp}, and; 

\item[ $(ii)$] (Pointwise invariance property) There exists $\bar{v}\in K$   such that \[-\Delta \bar{v}= D\Phi(\bar{u})= a(x)|\bar{u}|^{p-2}\bar{u},\] in the weak sense, i.e.,
\[\int_\Omega \nabla \bar v \cdot \nabla\eta \, dx  =\int_\Omega a(x)|\bar{u}|^{p-2}\bar{u}\eta\, dx, \qquad \forall \eta \in V.\]
\end{enumerate} 
 Then $\bar{u} \in K$ is  a weak solution of the equation 
\begin{equation}\label{eqthe}
\left\{\begin{array}{ll}
-\Delta u=  a(x)|u|^{p-2}u, &   x\in \Omega \\
u= 0, &    x\in  \partial \Omega.
\end {array}\right.
\end{equation}
\end{prop}

\begin{proof} Since $\bar{u}$ is a critical point of $E$,   it follows from Definition \ref{cp}   that 
\begin{equation}\label{p1}
 \CO{D \Phi(\bar{u})}{ \bar{u}-v}+ \Psi_K(v)- \Psi_K(\bar{u}) \geq 0, \qquad \forall v\in V.
\end{equation}
On the other hand, by (ii), there exists $\bar v \in K$ satisfying 
\begin{equation}\label{eqpp}
\left\{\begin{array}{ll}
-\Delta \bar  v =  a(x)|\bar{u}|^{p-2}\bar{u}, &   x\in \Omega \\
\bar v= 0, &    x\in  \partial \Omega,
\end {array}\right.
\end{equation}
in the weak sense. 
By setting $v= \bar{v}$ in (\ref{p1}) we obtain that 
\begin{eqnarray*}
\frac{1}{2}\int_{\Omega} |\nabla \bar{v}|^2dx  -\frac{1}{2}\int_{\Omega} |\nabla \bar{u}|^2 dx\notag  \geq  \int_{\Omega}a(x)|\bar{u}|^{p-2}\bar{u}(\bar{v}- \bar{u})\, dx
 = \int_{\Omega}\nabla\bar{v} \cdot \nabla(\bar{v}- \bar{u})\, dx
\end{eqnarray*}
where the last equality  follows from (\ref{eqpp}).
Therefore, 

\begin{equation}\label{p2}
\frac{1}{2}\int_{\Omega} |\nabla\bar{v}- \nabla \bar{u}|^2\, dx\leq 0,
\end{equation}
which implies that $\bar{u}= \bar{v}$. Taking into account that $\bar{u}= \bar{v}$ in (\ref{eqpp}) we have that $\bar u$ is a weak solution of 
(\ref{eqthe}).
\end{proof}

It is worth noting that  the condition $(ii)$ in Proposition \ref{prn} indeed shows that the triple $(\Psi, K, \Phi)$ satisfies the point-wise invariance condition  at $\bar u$ given in Definition \ref{def-invar} . In fact, it corresponds to the case where $G=0$.  This is why Proposition  \ref{prn} is a very particular case of the general Theorem \ref{vp}.

\subsection{The specific setting} 
We shall begin by  providing  some more background on quantities related to domains of double revolution that are essential in this work.   Assume $\Omega$ is a domain of double revolution and $ v$ is a function defined on $ \Omega$ that just depends on $ (s,t)$,  then one has 
\[ \int_\Omega v(x) dx = c(m,n) \int_{\widehat{\Omega}} v(s,t) s^{m-1} t^{n-1} ds dt,\]
where $c(m,n)$ is a positive constant depending on $n$ and $m.$ Note that strictly speaking we are abusing notation here by using the same name; and we will continuously do this in this article. Given a function $v$ defined on $ \Omega$ we will write $ v=v(s,t)$ to indicate that the function has this symmetry.     
Define  \[ H^1_{0,G}:=\left\{ u \in H^1_0(\Omega): gu =u \quad \forall g \in G \right\},\] where $G:=O(m) \times O(n)$ where $ O(k)$ is the orthogonal group in $ \IR^k$ and $ gu(x):=u( g^{-1} x)$. \\

To solve equations on domains of double  revolution one needs to relate the equation to a new one on $ \widehat{\Omega}$ defined in (\ref{omegahat}). Suppose $\Omega$  is a  domain of double revolution and $ f$ has is function defined on $\Omega$ with the same symmetry 
(ie. $ gf(x)= f(g^{-1} x)$ all $ g \in G$). 
Suppose that $u(x)$ solves 
\begin{equation} \label{eq_double_lin}
\left\{\begin{array}{ll}
-\Delta u(x) = f(x) &  \mbox{ in } \Omega, \\
u= 0 &   \mbox{ on } \pOm.
\end {array}\right.
\end{equation} Then $ u=u(s,t)$ and $u$ solves 
\begin{equation} \label{eq_double_lin_st}
-u_{ss}-u_{tt}- \frac{(m-1) u_s}{s}-\frac{(n-1) u_t}{t} = f(s,t)   \mbox{ in } \widehat{\Omega}, 
\end{equation} with $u=0$ on $ (s,t) \in \partial \widehat{\Omega} \backslash ( \{s=0\} \cup \{t=0\} )$.  If $u$ is sufficiently smooth then $ u_s =0 $ on $\partial \widehat{\Omega} \cap \{s=0\}$ and $u_t=0$ on $ \partial \widehat{\Omega} \cap \{t=0\}$ after considering the symmetry properties of $u$.  \\

We now pick our convex set $K$.  We mention that this idea of restricting functions to ones which are monotonic in an angle, to improve compactness, is coming from the work of \cite{Weth_annulus}.
  We define $K$ by 
\begin{equation}\label{cK} K=K(m,n):=\left\{ 0 \le u \in H_{0,G}^1(\Omega):  su_t-tu_s \le 0 \mbox{ a.e. in  } \widehat{\Omega} \right\}, 
\end{equation} and note we can rewrite $K$ as functions $u$ such that if we write $(s,t)$  in terms of polar coordinates we have $u_\theta \le 0$ in $ \widetilde{\Omega}$ defined in (\ref{omegahat+}). 
 In our approach there are generally two crucial ingredients when applying Theorem \ref{vp} to supercritical problems.  Part (i) relies on picking $K$ appropriately so as to gain  some needed compactness in order to prove  existence of a critical point for the nonsmooth functional $E_K=\Psi_K-\Phi.$   In part (ii) we need to be able to prove the needed `point-wise invariance property' to land on a true critical point of $E$ instead of $E_K.$ 


\begin{thm}(Imbeddings for annular  domains) \label{compact_no_mono} Let $\Omega$ denote an annular region in $ \IR^N$ (recall we are always  assuming that $\Omega$ is a domain of double revolution). 

\begin{enumerate}  \item (Imbedding without monotonicity) Suppose $\Omega$ has no monotonicity and  \[ 1 \le p < \min \left\{ \frac{2(n+1)}{n-1},  \frac{2(m+1)}{m-1} \right\}.\] Then 
$ H^1_{0,G}(\Omega) \subset \subset  L^p(\Omega)$ with the obvious interpretation in the case of $m=n=1$.

\item (Imbedding with monotonicity) Suppose $ \Omega$ is an annular domain with monotonicity, $ n \le m$ and   
\[ 1 \le p<  \frac{2(n+1)}{n-1} =\max  \left\{ \frac{2(n+1)}{n-1},  \frac{2(m+1)}{m-1} \right\}.\]
Then $K \subset \subset L^p(\Omega)$  with the obvious interpretation if $n=1$.
\end{enumerate} 
\end{thm}

\begin{proof}  1.     First note that to prove the desired result it is sufficient to prove  $ \| u\|_{L^p(\Omega)} \le C \| u\|_{L^2(\Omega)} + C \| \nabla u\|_{L^2(\Omega)}$ for all $ u \in H^1_{0,G}(\Omega)$.   Take $ T>0$ large such that $ \widehat{\Omega} \subset (0,T)^2=:Q$ 
and given a function $u \in H^1_{0,G}(\Omega)$ (hence $ u=u(s,t) $ for $ (s,t) \in \widehat{\Omega}$) we extend $u$ to all of $Q$ by setting $u$ to be zero outside $ \widehat{\Omega}$.  Fix $ \E>0$ small such that $ (0,\E)^2 \cap \widehat{\Omega} = \emptyset$ and consider $ \gamma_\E(s,t)=\gamma_\E(s)$ denote some smooth cut off function with $ \gamma_\E=0$ for $ s<\E$ and $ \gamma_\E=1$ for $ s>2\E$.  We claim there is some $ C=C(p)$ such that 
\begin{equation} \label{no_s}
\left( \int_Q | \gamma_\E(s) u(s,t)|^p t^{n-1} ds dt \right)^\frac{2}{p} \le C \int_Q (\gamma_\E u)^2 t^{n-1} ds dt + C \int_Q | \nabla_{s,t} (\gamma_\E u)|^2 t^{n-1} ds dt,
\end{equation}
for all $ u \in H^1_{0,G}(\Omega)$ for all $ 2<p<\frac{2(n+1)}{(n+1)-2}$ for $ n \ge 2$ (for $n=1$ one gets the result for all $p<\infty$).  To see the result we are considering the function $ (s,t) \mapsto \gamma_\E(s) u(s,t)$ to be a function $ (s,t) \in (0,T) \times B_T \subset  \IR^{n+1}$ where $B_T$ is the ball of radius $T$ centred at the origin in $ \IR^n$ and the function is radial in $t$ and hence the weight $ t^{n-1}$.  We now need reinsert the weight $ s^{m-1}$ into the inequality.   Because of the cut off $ \gamma_\E$ there is some $ C=C(\E,p)$ such that the right hand side of (\ref{no_s}) is bounded above by
 \[C \int_Q (\gamma_\E u)^2 t^{n-1} s^{m-1} ds dt + C \int_Q | \nabla_{s,t} (\gamma_\E u)|^2 t^{n-1} s^{m-1} ds dt,\] and the left hand side of (\ref{no_s}) is an upper bound for the same term with the added weight $ s^{m-1}$ and hence we arrive at 
\begin{equation} \label{no_s_weight}
\left( \int_Q | \gamma_\E(s) u(s,t)|^p d \mu\right)^\frac{2}{p} \le C \int_Q (\gamma_\E u)^2 d \mu  + C \int_Q | \nabla_{s,t} (\gamma_\E u)|^2 d \mu,
\end{equation} where $ d \mu(s,t)=s^{m-1} t^{n-1} ds dt$.    Note we can bound the right hand side above by $ C_{1,\E} \int_\Omega | \nabla u(x)|^2 dx$ for some constant $ C_{1,\E}$ independent of $u$ after considering the $L^2(\Omega)$ Poincare inequality.     One can similarly obtain 
\[ \left( \int_Q | \gamma_\E(t) u(s,t)|^p d \mu \right)^\frac{2}{p} \le C \int_\Omega | \nabla u(x)|^2dx,\] for all $ u \in H^1_{0,G}(\Omega)$ where we are again extending $u$ and here $\gamma_\E$ is a cut off in the $t$ direction,  and where $ 2<p < \frac{2(m+1)}{(m+1)-2}$.    Since $ (0,0) \notin \overline{ \widehat{\Omega}}$ we can combine the results to see we have a continuous imbedding for 
\[ 2<p<\min \left\{ \frac{2(n+1)}{(n+1)-2},  \frac{2(m+1)}{(m+1)-2} \right\}.\]  To obtain the compact imbedding one uses an $L^p$ interpolation inequality along with compactness in, for instance, $L^1(\Omega)$. \\

\noindent  2.  We assume that $1\leq n\leq m$ and $p<\frac{2(n+1)}{(n-1)}$.  By using the polar coordinates $s=r \cos \theta$ and $t=r \sin \theta$ we have that 

\begin{eqnarray}
\int_{\widehat{\Omega}}  u(s,t)^p s^{m-1} t^{n-1} ds dt=\int_{0}^{\pi/2}\int_{g_1}^{g_2} r^{m-1}\cos^{m-1}(\theta) r^{n-1} \sin^{n-1}(\theta) u(r,\theta)^pr\, dr \, d \theta.
\end{eqnarray} 
For $\theta \in [\pi/3, \pi/2]$ we have that
$\sin(\theta) \leq c \sin(\theta-\pi/4)$ for some constant $c>0$.  Thus, considering the monotonicity properties of $g_1$, $g_2$  and $\theta \mapsto u(r, \theta)$ we obtain that 
\begin{eqnarray*}
\int_{\pi/3}^{\pi/2}\int_{g_1(\theta)}^{g_2(\theta)} r^{m-1}\cos^{m-1}(\theta) r^{n-1} \sin^{n-1}(\theta) u(r,\theta)^pr\, dr \, d \theta\\ \leq c^{n-1}\int_{\pi/3}^{\pi/2}\int_{g_1(\theta-\pi/4) }^{g_2(\theta-\pi/4) } r^{m-1}\cos^{m-1}(\theta-\pi/4) r^{n-1} \sin^{n-1}(\theta-\pi/4) u(r,\theta-\pi/4)^pr\, dr \, d \theta\\
=c^{n-1}\int_{\pi/12}^{\pi/4}\int_{g_1(\theta)}^{g_2(\theta)} r^{m-1}\cos^{m-1}(\theta) r^{n-1} \sin^{n-1}(\theta) u(r,\theta)^pr\, dr \, d \theta.
\end{eqnarray*}
Thus,  there is a constant $C_1>0$ such that 
\begin{eqnarray*}
\int_{0}^{\pi/2}\int_{g_1}^{g_2} r^{m-1}\cos^{m-1}(\theta) r^{n-1} \sin^{n-1}(\theta) u(r,\theta)^pr\, dr \, d \theta\\ \leq C_1 \int_{0}^{\pi/3}\int_{g_1}^{g_2} r^{m-1}\cos^{m-1}(\theta) r^{n-1} \sin^{n-1}(\theta) u(r,\theta)^pr\, dr \, d \theta
\end{eqnarray*}
On the other hand,
\begin{eqnarray*}
 \int_{0}^{\pi/3}\int_{g_1}^{g_2} r^{m-1}\cos^{m-1}(\theta) r^{n-1} \sin^{n-1}(\theta) u(r,\theta)^pr\, dr \, d \theta=\int_{\{\widehat \Omega, \,\, s\geq \beta\}}  u(s,t)^p s^{m-1} t^{n-1} ds dt
\end{eqnarray*}
for some positive constant $\beta$. Therefore, 
\begin{eqnarray*}
\left (\int_{\{\widehat{\Omega}, \,\, s\geq \beta\}}  u(s,t)^p s^{m-1} t^{n-1} ds dt\right)^{2/p} \leq  C_2 \left (\int_{\{{\widehat{\Omega}}, \,\, s\geq \beta\}}  u(s,t)^p  t^{n-1} ds dt\right)^{2/p}
\end{eqnarray*}
Thus, by part 1),
\begin{eqnarray*}
 \left (\int_{\{\widehat{\Omega}, \,\, s\geq \beta\}}  u(s,t)^p  t^{n-1} ds dt\right)^{2/p} &\leq & C_3\int_{\{{\widehat{\Omega}}, \,\, s\geq \beta\}}  (u^2+u_s^2+u_t^2)  t^{n-1} ds dt\\ & \leq & C_4 \int_{\{{\widehat{\Omega}}, \,\, s\geq \beta\}}  (u^2+u_s^2+u_t^2)  t^{n-1} s^{m-1} ds dt \\ &\leq & C_4 \int_{{\widehat{\Omega}}}  (u^2+u_s^2+u_t^2)  t^{n-1} s^{m-1} ds dt=
C_5\|u\|^2_{H^1(\Omega)}.
\end{eqnarray*}
\end{proof}

 We remark that if $N$ is even and $ m=n=\frac{N}{2}$  then monotonicity does not improve the imbedding.  This is not surprising when one examines the proof of the improved imbedding with monotonicity.  

\begin{coro} \label{torus_imbed} (Imbeddings for toroidal domains) Let $\Omega$ denote an toroidal domain in $ \IR^N$ (recall we are always  assuming that $\Omega$ is a domain of double revolution). Then $H_{0,G}^1(\Omega) \subset \subset L^p(\Omega)$ for $ 1 \le p < \frac{2(n+1)}{n-1}$.
\end{coro}

\begin{proof} This result follows by using the proof of part 1 of Theorem  \ref{compact_no_mono} and noting that the term $ s^{m-1}$ does not play a role since $ s$ is bounded away from zero on $ \widehat{\Omega}$. 
\end{proof}

In the next lemma we discuss the existence of a critical point for the  the non-smooth functional $E_K$ by means of Theorem \ref{MPT}.
\begin{lemma}\label{PS0}
The functional $E_K$ defined in (\ref{E}) satisfies the mountain pass geometry and  (PS) compactness condition (Definition \ref{psc}).
\end{lemma}
\begin{proof}
Note first that  that $K$ is a weakly closed convex cone in $H_0^1(\Omega)$.
It follows from Theorem \ref{compact_no_mono}  that there exists a constant $C$ such that 
\begin{equation}\label{equiv}\|u\|_{H^1}\leq \|u\|_V \leq C \|u\|_{H^1}, \qquad \forall u \in K. \end{equation}

Both the mountain pass geometry and  (PS) compactness condition follow from the standard arguments together with inequality (\ref{equiv}). Here,  for the conveience of the reader,  we sketch the proof for the (PS) compactness condition.
Suppose that $\{u_n\}$ is a sequence in $K$ such that $E(u_n)\rightarrow c\in \mathbb{R}$, $\E_n \rightarrow 0$ and
\begin{equation}\label{od}
 \Psi_K(v)- \Psi_K(u_n)+ \langle D\Phi(u_n), u_n-v \rangle\geq -\E_n \|v- u_n\|_V, \qquad \forall v\in V.
\end{equation}
We must show that $\{u_n\}$ has a convergent subsequence in $V$.  Firstly,  we prove that $\{u_n\}$ is bounded in $V$. 
Note that 
   since $E(u_n)\rightarrow c$, then  for large values of $n$ we have 
\begin{equation}\label{od0}
 \frac{1}{2}\| u_n\|^2_{H^1}- \frac{1}{p}\int_{\Omega}  a(x)|u|^p dx\leq c+1.
\end{equation}
Note that 
\[\langle D\varphi(u_n),  u_n\rangle = \int_{\Omega} a(x) u_n(x)^{p} dx.\]
Thus, by setting  $v= ru_n$ in (\ref{od}) with $r=1+1/p$ we get 
\begin{equation}\label{od1}
\frac{(1-r^2)}{2} \|u_n\|^2_{H^1}+ (r-1)\int_{\Omega}a(x)|u_n|^p dx\leq \E_n (r-1) \| u_n\|_V.
\end{equation}
Adding up  (\ref{od1}) and  (\ref{od0}) yields that 
\begin{equation*}
 \|u_n\|^2_{H^1}\leq C_0(1+  \| u_n\|_{V}),
\end{equation*}
for some constant $C_0>0.$
Therefore, by considering (\ref{equiv}),   $\{u_n\}$ is bounded in $H^1.$  Using standard results in Sobolev spaces, after passing to a subsequence if necessary, there exists $\bar u \in H^1$ such that  $u_n\rightharpoonup \bar{u}$ weakly in $H^1$,   $u_n\rightarrow  \bar{u}$ a.e..  Also according to Theorem \ref{compact_no_mono},  from boundedness of $\{u_n\}\subset K$  in $H^1$, one can deduce that  the strong  convergence of $u_n$ to $\bar u$ in  $L^p.$ 
Now in (\ref{od}) set $v= \bar{u}$:
\begin{equation}\label{oder1}
\frac{1}{2} ( \|\bar{u}\|^2_{H^1} - \|u_n\|^2_{H^1}) +\int_{\Omega}a(x)|u_n|^{p-1} ( u_n- \bar{u}) dx
\geq -\E_n \|u_n- \bar{u}\|_V.
 \end{equation}
Therefore, it follows from  (\ref{oder1})  that 
\[
\frac{1}{2} (      \limsup_{n\rightarrow \infty} \|u_n\|^2_{H^1}  -     \|\bar{u}\|^2_{H^1}  ) \leq 0.
\]
The latter  yields that 
$$u_n\rightarrow \bar{u}\quad \text{strongly in }\quad V$$
as desired.
\end{proof}

We shall now verify that the triple $(\Psi, \Phi, K)$ satisfies the point-wise invariance condition as stated in Definition \ref{def-invar} and more specifically in part $(ii)$ of Proposition 
\ref{prn}.


 \begin{prop} (Pointwise invariance property  on $K$) \label{pointwise_double} Suppose $\Omega$ is an annular domain with monotonicity and $ a$ satisfies $(\mathcal{A})$. Let $ u \in K$ and suppose $ v$ solves
 \begin{equation} \label{linear_new_st}
\left\{\begin{array}{ll}
-\Delta v = a(x) u^{p-1} &  \mbox{ in } \Omega, \\
v= 0 &   \mbox{ on } \pOm.
\end {array}\right.
\end{equation} Then $ v \in K$. 
\end{prop}

\begin{proof} Fix $ u \in K$ and for $k \ge 1$ we set $u_k(x)= \min\{ u(x), k\}$ and note that $u_k \in H_0^1(\Omega)$.  Recall that $u$ has the symmetry $ u \in H^1_{0,G}(\Omega)$ and the cut off does not effect the symmetry and hence $u_k \in  H^1_{0,G}(\Omega)$.   One should also note that the monotonicity of $u$ is preserved under the cut off;  this may be more apparent if one writes $u$ in the variables $(s,t)$ and then writes this in terms of polar coordinates $ s=r \cos(\theta), t =  r \sin(\theta)$.    So we know consider solving 
\begin{equation} \label{linear_new_cut-off}
\left\{\begin{array}{ll}
-\Delta v = a(x) u_k^{p-1} &  \mbox{ in } \Omega, \\
v= 0 &   \mbox{ on } \pOm.
\end {array}\right.
\end{equation}  
Considering the associated energy on $H^1_{0,G}(\Omega)$ we see the existence of a unique $ 0 \le v^k \in H^1_{0,G}(\Omega) \cap C^{1,\alpha}(\overline{\Omega})$ which solves (\ref{linear_new_st}).  Set $ w^k= sv^k_t - t v^k_s$.  Then a computation shows that
\begin{equation} \label{w_eq_10}
-w^k_{ss}-w^k_{tt} -\frac{(m-1) w^k_s}{s} - \frac{(n-1) w^k_t}{t} + \frac{(m-1) w^k}{s^2} +\frac{(n-1) w^k}{t^2}  = H \quad (s,t) \in \widehat{\Omega}
\end{equation} 
where 
\[H:= u_k^{p-1} (s a_t - t a_s)+ (p-1) a u_k^{p-2} (s (u_k)_t - t (u_k)_s) ,\] and note $H \le 0$ in $ \widehat{\Omega}$.    Note that this problem behaves like a two dimensional problem away from the sets $ \{s=0\}$ and $ \{t=0\}$.    Note that if $v^k$ is sufficiently smooth that by the symmetry of $v^k$ we expect $v_s^k=0$ and $ v_t^k=0$ on  the portions of $ \partial \widehat{\Omega}$ that correspond to $ \{s=0\}$ and $ \{t=0\}$ respectively.   Using this we see  that we should have $ w=0$ on these portions of the boundary.     
Note that if $ m$ and $n \ge 2$ then both these portions of the boundary are dimension less than or equal $N-2$ and hence trying to utilize trace arguments for Sobolev functions may be a delicate issue.  This is  the reason we considered the slightly smoothed version of $ u \in K$ given by $ u_k$.  Since $ v^k=v^k(s,t)$ is the restriction to the first quadrant of $(x_1,x_{m+1})$ plane of an even $ C^{1,\alpha}$ function in $x_1$ and $ x_{m+1}$ we see that $ v_s^k, v_t^k \in C^{0,\alpha}( \overline{\Omega})$. 
This is sufficient regularity for $ v_s^k$ and $v_t^k$ to give the desired boundary conditions on $ \{s=0\}$ and $ \{t=0\}$ portions of $ \partial \widehat{\Omega}$ respectively and hence we see that $ w^k=0$ on $ \partial \widehat{\Omega} \cap \left( \{s=0\} \cup \{t=0\} \right)$. 
Following the computation in \cite{Cabre_double} we have 
\begin{equation} \label{cab_form}
v^k_s= \sum_{i=1}^m v^k_{x_i} \frac{x_i}{s} \quad \mbox{ and } \quad  v_t^k = \sum_{i=m+1}^N v^k_{x_i} \frac{x_i}{t}.
\end{equation} 
 
  By the symmetry and monotonicity assumptions on the domain along with the fact $ v^k \ge 0$ we will be able to show that $ w^k \le 0$ on the remaining portion of $ \partial \widehat{\Omega}$;  we think this is more apparent if the equation is written in polar coordinates and so we will examine this boundary condition after doing that.   Using a computation similar to the above one we see that $H \in L^2(\Omega)$.

In terms of polar coordinates (recall by this we are taking $ s=r \cos(\theta), t = r \sin(\theta)$) we have 

\begin{eqnarray} \label{polar_w}
H &= &-w^k_{rr}-\frac{w^k_r}{r}- \frac{w^k_{\theta \theta}}{r^2} -\frac{(N-2) w^k_r}{r} + \frac{w^k_\theta}{r^2} \left( (m-1) \tan(\theta) -(n-1) \cot(\theta) \right) \nonumber \\
&&+ \frac{(m-1)w^k}{r^2 \cos^2(\theta)}+ \frac{(n-1) w^k}{r^2 \sin^2(\theta)}  \quad \mbox{ in } (r,\theta) \in \widetilde{\Omega},
\end{eqnarray} 
and recall that $ g_1$ is increasing and $g_2$ is decreasing on $(0, \frac{\pi}{2})$, which is the assumption on the domain we are making.     Now recall we made the earlier claim that $ w^k \le 0$ on $ \pOm$.  Note that we can write $ w^k= sv_t^k- t v_s^k = v_\theta$ for $ (r,\theta) \in \widetilde{\Omega}$.   Using the fact that $v=0$ on $ \pOm$ and fact $ v \ge 0$ in $\Omega$ and the assumptions on $g_i$ we see that $ v_\theta \le 0$ on the boundaries of $ \partial \widetilde{\Omega}$ corresponding to $ r = g_1(\theta)$ and $ r=g_2(\theta)$.    So the idea now is that we would like to apply the maximum principle to either 
(\ref{w_eq_10}) or (\ref{polar_w}); if we could do this then it should imply that $ w^k \le 0$ and we would have that $ v^k \in K$.    One problem with applying the maximum principle is that there are multiple singularities in the equation which cause problems.    
Before we complete the proof we would like to write the equation for $ w^k$ in terms of $ x$.   A computation shows that $ w^k(x)$ satisfies  
\begin{equation} \label{w_x_var}
     -\Delta w^k(x) + \frac{(m-1) w^k}{x_1^2 + \cdot \cdot \cdot + x_m^2} + \frac{(n-1) w^k}{ x_{m+1}^2+ \cdot \cdot \cdot + x_N^2}  =H \quad \mbox{ in } \Omega \backslash \mathcal{N},
     \end{equation} where $\mathcal{N}:= \{s=0\} \cup \{t=0\}$.  Depending on the values of $m$ and $n$ one can probably show the equation is satisfied in the sense of distributions on the full $ \Omega$,  but this won't effect our approach.   To apply the maximum principle we will consider (\ref{w_eq_10}).   First note that $H \in L^2_{loc}( \overline{ \widehat{\Omega}} \backslash \mathcal{N})$ and hence we have $ w^k \in W^{2,2}_{loc}( \widehat{\Omega} \backslash \mathcal{N})$.  
     For $ \E>0$ small consider $ \psi(s,t):= (w^k(s,t)-\E)_+$ and note $ \psi=0$ near $ \partial \widehat{\Omega}$.   This is sufficient regularity to mulitply (\ref{w_eq_10}) by $ \psi$ and integrate (note since $ \psi=0$ near $ \mathcal{N}$ there are no issues with any of the possible singularities) to arrive at 
    \[\int_{\widehat{\Omega}} | \nabla_{s,t} (w^k-\E)_+|^2 d \mu(s,t)
     +   \int_{\widehat{\Omega}} \left( \frac{m-1}{s^2}+ \frac{ n-1}{t^2} \right) w^k (w^k-\E)_+ d \mu(s,t) = \int_{\widehat{\Omega}} H (w^k-\E)_+ d \mu(s,t)\le 0
     \]
     where $ d \mu(s,t)= s^{m-1} t^{n-1} ds dt$.  This is sufficient to show that $ (w^k-\E)_+=0$ and hence $ w^k \le \E$ on $ \widehat{\Omega}$ for all $ \E>0$ and hence $ w^k \le 0$ in $ \widehat{\Omega}$.     
     
     We now show that $v^k$ is bounded in $H_0^1(\Omega)$.  To see this test (\ref{linear_new_cut-off}) on $v^k$ to arrive at 
     \begin{eqnarray*}
     \int_\Omega | \nabla v^k(x)|^2 dx &=& \int_\Omega a u_k^{p-1} v^k dx\\
     & \le & C_a  \|u_k^{p-1} \|_{L^{p'}} \|v_k\|_{L^p}    \\
     &=&  \| u_k \|_{L^p}^{p-1} \|v^k \|_{L^p}  \\
     & \le & C \| \nabla u_k \|_{L^2}^{p-1} \| \nabla v^k\|_{L^2} 
     \end{eqnarray*} where we have used the improved imbedding in the last line and the fact that $u_k, v^k \in K$.   Using the fact that $ u \in K$ and hence $ \| \nabla u_k\|_{L^2}$ is bounded we see that $\{v^k\}_k$ is bounded in $H_0^1(\Omega)$ and hence after passing to a subsequence we have $ v^k \rightharpoonup v$ in $H_0^1(\Omega)$ for some $ v$.   Using standard arguments we can pass to the limit in  (\ref{linear_new_cut-off}) to see that $v \in H_0^1(\Omega)$ satisfies (\ref{linear_new_st}).   Also notes it clear that $v \in H^1_{0,G}(\Omega)$.  Also note that $ u_k^{p-1} \rightarrow u^{p-1} $ in $L^{p'}(\Omega)$ and hence by passing to another subsequence we have $ v^k \rightarrow v$ in $ W^{2,p'}(\Omega)$ and hence we can assume $ \nabla v^k \rightarrow \nabla v$ in $L^{p'}(\Omega)$ and a.e. in $ \Omega$.   We now set $ w:=sv_t-t v_s$.   Writing $ v^k_s$ and $ v_s$ as in (\ref{cab_form}) we see that $ (v^k)_s \rightarrow v_s$ a.e. in $ \Omega$ and hence we have the same for a.e. $ (s,t) \in \widehat{\Omega}$ and similarly for $ (v^k)_t \rightarrow v_t$ for a.e. $(s,t) \in \widehat{\Omega}$ and hence $ w^k \rightarrow w$ for a.e. $ (s,t) \in \widehat{\Omega}$ and hence $ w \le 0$ in $\widehat{\Omega}$. This gives the desired monotonicity of $v$ and hence $ v \in K$.  
     
     \end{proof}

\noindent 
\textbf{The completion of the proof of Theorem \ref{main}.} We shall first prove the existence of a weak solution for parts {\it 1-3}, and then we show that this solution has to be a classical one.\\

\noindent 
{\it 1.}  It follows from Lemma \ref{PS0} that $E_K$ has a critical point $\bar u \in K,$  with 
 $E_K(\bar u)= c>0$ where the critical value $c$ is characterized by
$$c= \inf_{\gamma\in \Gamma}\max_{t\in [0,1] }E_K[\gamma(t)],$$
where \[\Gamma= \{\gamma\in C([0,1], V) : \gamma(0)=0\neq \gamma(1), E_K(\gamma(1))\leq 0\}.\]
Since  $E_K(\bar u)>0$,  we have that $\bar u$ is non-zero. It also follows from
Proposition \ref{pointwise_double} that there exists $\bar v \in K$ satisfying 
 \[-\Delta \bar{v}= D\Phi(\bar{u})= a(x)|\bar{u}|^{p-2}\bar{u}.\]
 It now follows from Proposition \ref{prn} that $\bar u \in K$ is a solution of 
\begin{equation}\label{pscb}
\left\{\begin{array}{ll}
-\Delta u=  a(x)|u|^{p-2}u, &   x\in \Omega \\
u= 0, &    x\in  \partial \Omega.
\end {array}\right.
\end{equation} \\

\noindent
{\it 2.}  It follows from Theorem \ref{compact_no_mono} that $ H^1_{0,G}(\Omega) \subset \subset  L^p(\Omega)$  is compact for 
\[ 1 \le p < \min \left\{ \frac{2(n+1)}{n-1},  \frac{2(m+1)}{m-1} \right\}.\]  
It can be easily deduced that the functional 
$$E(u):= \frac{1}{2}\int_{\Omega} |\nabla u|^2 dx- \frac{1}{p}\int_{\Omega}  a(x)|u_+|^p dx,$$
satisfies the assumptions of the mountain pass theorem  on $H^1_{0,G}(\Omega),$  and therefore has a  non-negative  critical point on $H^1_{0,G}(\Omega)$.  It now follows from the principle of  Symmetric Criticality  that every critical point of $E$ is a solution of 
(\ref{pscb}).\\

\noindent
{\it 3.} The proof follows the same basic approach as the case of the annular domain with monotonicity.  The main difference is one has to deal with a different domain in the $ (\theta,r)$ coordinates.  The monotonicity of $g_i$ is needed for the same reason as the case of the annular domains.  Because we are not obtaining an increased range of $p$ for the monotonic case as compared to the nonmonotonic case we don't think its worth going into the exact details of the proof.  
\\

\noindent
{\it Regularity of the solution.}  We have shown the existence of a nonzero nonnegative solution $u \in K \subset H^1_{0,G}(\Omega)$ of (\ref{pscb}) in the proof of part {\it 1}. To complete the proof we need to show the solution is a classical solution and strictly positive.  First note that we can apply the Strong Maximum Principle to see the solution is positive after we have sufficient regularity.  Here we assume we are in the case of the annular domain with monotonicity and take $ 2 \le n \le m $ (if $n=1$ there is no iteration needed). Set $ t_0:=\frac{n+1}{n-1} - \frac{p-2}{2}$   and for $ k \ge 0$ set 
\[ t_{k+1} = \frac{(n+1) t_k}{n-1}- \frac{p-2}{2}.\] Note that $ t_0> \frac{(n-1)(p-2)}{4}$ and examining the cob-web shows that $ t_k$ increases to infinity.   Also note that $ t_0>1$. We now prove the following iteration step: suppose $ k \ge 0$ and 
\begin{equation} \label{iteration_assump}
\int_\Omega a(x)u^{p+2t_k-2} dx =C_k<\infty,
\end{equation}
then $u^{t_k} \in K$ and  
\[ \int_\Omega a(x) u^{p+2t_{k+1} -2} dx <\infty.\]  First recall that if $ 0 \le u \in H_0^1(\Omega)$ and $ -\Delta u \ge 0$  then $ u^{\gamma} \in H_0^1(\Omega)$ for all $ \frac{1}{2}<\gamma \le 1$. Since $ t_k \ge 1$ the only issue about $ u^{t_k} \in H_0^1(\Omega)$ is whether it grows too quick. For large $m$ we set 
 \begin{equation} 
\phi(x)=\left\{\begin{array}{ll}
u(x)^{2t_k-1}  &  \mbox{ if } u(x)<m, \\
m^{2t_k-1}     &   \mbox{ if }u(x) \ge m, \\
\end {array}\right.
\end{equation} and hence $ \phi \in K$ and is a suitable test function to put into the weak formulation of the pde.  This gives 
\[ (2t_k-1) \int_{\{u<m\} } u^{2t_k-2} | \nabla u|^2 dx = \int_{ \{u<m\}}a(x) u^{p+2t_k-2} dx + \E_m,\] where 
\[ \E_m=\int_{ \{ u>m\}} a(x)u^{p-1} m^{2t_k-1} dx.\]  Using the assumption (\ref{iteration_assump})  one sees that $\E_m \rightarrow 0$ as $ m \rightarrow \infty$ and hence we arrive at 
\[ (2t_k-1) \int_\Omega u^{2t_k-2} | \nabla u|^2 dx = \int_\Omega a(x) u^{p+2t_k-2} dx,\]  and hence $ u^{t_k} \in H_0^1(\Omega)$ since the integral on the left is equal to $ t_k^{-2} (2t_k-1) \| \nabla u^{t_k}\|_{L^2}^2$.  Its clear that in fact one has 
 $u^{t_k} \in K$. Rewriting the above and using the improved imbedding for functions in $K$ we arrive at 
\[ C \left( \int_\Omega u^\frac{t_k 2(n+1)}{n-1} dx \right)^\frac{n-1}{n+1} \le \int_\Omega | \nabla u^{t_k}|^2dx = \frac{t_k^2}{2t_k-1} \int_\Omega a(x) u^{p+2t_k-2} dx,\] where $C=C(\Omega)$. Now note that $ \frac{t_k 2(n+1)}{n-1}= p+2t_{k+1}-2$ and hence we have proved the iteration step since $ a(x)$ is bounded and can be inserted into the interal on the left.  To show the initial step test the equation on $u$ and use the improved imbedding to see \[ C \left( \int_\Omega u^\frac{2(n+1)}{n-1} dx \right)^\frac{n-1}{n+1} \le \int_\Omega a(x) u^p dx,\] and then note that $ t_0$ satisfies $ p+2t_0-2= \frac{2(n+1)}{n-1}$.   So using the above iteration we see that after a finite number of iterations we have that $ u^{p-1} \in L^q(\Omega)$ for some $ q>\frac{N}{2}$ and then using elliptic regularity and the classical Sobolev imbedding we have that $ u$ is H\"older continuous.  We can then switch to the Schauder regularity theory and proceed to show that $u$ is a classical solution.   As mentioned earlier, in the case of $n=1$ an iteration argument is not needed.  This follows after examing the imbedding of $K$ for $n=1$.   The other cases are similar and we omit their proofs.

\hfill $\square$

\section{Nonradial solutions when $\Omega$ is an annulus.}

In this section we  discuss the case when $a(x)=a(|x|)$ is radial,  and $\Omega$ is an annulus,  
 that is $\Omega=\{x:\,  R_1< |x|<R_2\}, $ 
 \begin{equation} \label{eqzvv}
\left\{\begin{array}{ll}
-\Delta u = a(|x|) u^{p-1} &  \mbox{ in } \Omega, \\
u>0    &   \mbox{ in } \Omega, \\
u= 0 &   \mbox{ on } \pOm.
\end {array}\right.
\end{equation}
 We shall prove that the solution obtained in Theorem \ref{main}  is nonradial when   radii $R_1, R_2$ satisfy  certain conditions.\\

We require some preliminaries before proving our theorems for the radial domain.  Consider the variational formulation of an eigenvalue problem given by 
\begin{equation} \label{var_mu}
    \mu_1=\inf_{\psi \in H^1_{loc}(0, \frac{\pi}{2})}\Big \{ \int_0^{\frac{\pi}{2}} |\psi'(\theta)|^2 \omega(\theta) \,d\theta; \quad    \int_0^{\frac{\pi}{2}} |\psi(\theta)|^2\omega(\theta) \,d\theta=1, \int_0^{\frac{\pi}{2}} \psi(\theta)\omega(\theta) \,d\theta=0     \Big\},
\end{equation} where $\omega(\theta):=\cos^{m-1}(\theta) \sin^{n-1}(\theta)$ and suppose $ \psi_1$ satisfies the minimization problem.  Then $ (\mu_1,\psi_1)$ satisfies 
\begin{equation} \label{eigen_p}
\left\{\begin{array}{ll}
 -\partial_\theta (\omega(\theta) \psi_1'(\theta) )= \mu_1 \omega(\theta)\psi_1(\theta)&  \mbox{ in } (0, \frac{\pi}{2}), \\
\psi'(\theta)>0   &   \mbox{ in } (0, \frac{\pi}{2}), \\
\psi_1'(0)= \psi_1'(\frac{\pi}{2})=0, & 
\end {array}\right.
\end{equation}  
and note $(\mu_1,\psi_1)$ is the second eigenpair, the first eigenpair is given by $ (\mu_0,\psi_0)=(0,1)$.  Note in this problem one can find an explicit solution given by 
   \[ \mu_1=2N, \quad \psi_1(\theta)= \frac{m-n}{N} - \cos(2 \theta),\] and note we can apply  Sturm–Liouville theory and count the number of zeros of $ \psi_1$ to see it is in fact the second pair.

We are now  ready to prove the existence of the non-radial solutions, but before we do, we make some comments regarding the proof.  Our proof does not use the fact that $ \psi_1$ satisfies (\ref{eigen_p}) but rather it utilizes the fact that $\psi_1$ is where the minimum is attained (of course in this case they are equivalent).  If $ \mathcal{B}_1$ is the constraint set from the definition of $ \mu_1$ in  (\ref{var_mu}) set $ \mathcal{B}_1^+:=\{ \psi \in \mathcal{B}_1: \psi' \ge 0 \mbox{ a.e. in } (0,\frac{\pi}{2}) \}$ and let $ \mu_1^+$ denote the associated quantity if we were to minimize over $\mathcal{B}_1^+$.  The same proof would carry over if we replaced  $ (\mu_1,\psi_1)$ with $ (\mu_1^+, \psi_1^+)$ where $ \psi_1^+$ is the associated minimizer which a  straightforward argument show the existence of.  Note apriori the pointwise constraint on the gradient can  significantly complicate finding  associated  Euler-Lagrange equation and so not needing to use the equation can be a benefit. \\

\noindent
\textbf{Proof of Theorem \ref{nnr}.} Let us assume  the solution  $u$ of (\ref{eqzvv}) obtained in Theorem \ref{main}  is radial. Recall  that  $E_K(u)= c>0$ where the critical value $c$ is characterized by
$$c= \inf_{\gamma\in \Gamma}\max_{\tau\in [0,1] }E_K[\gamma(\tau)],$$
where \[\Gamma= \{\gamma\in C([0,1], V) : \gamma(0)=0\neq \gamma(1), E_K(\gamma(1))\leq 0\}.\]
For the sake of simplifying the notations, we use $E$ instead of $E_K$ in the rest of the proof.
Set $v(r,\theta)=u(r)\psi(\theta)$ where $\psi(\theta)=\frac{m-n}{N} - \cos(2 \theta)$ being a solution of (\ref{eigen_p}) with $ \mu_1=2N. $ We first show that
\begin{equation}\label{qw}
 \int_\Omega |\nabla v|^2 \, dx-(p-1)\int_\Omega|a(|x|)u|^{p-2}v^2\,dx <0.
\end{equation}
To this end we need to show that $M(u,v)<0$ where
\begin{equation}\label{qw00}
  M(u,v):=\int_{\widehat\Omega} s^{m-1}t^{n-1}(v_t^2+v_s^2) \, ds\, dt-(p-1)\int_{\widehat\Omega}s^{m-1}t^{n-1} a(s,t)u^{p-2}v^2\,ds\,dt <0.
\end{equation}
Note first that it follows from the equation $-\Delta u=a(r)u^{p-1}$ that 
\begin{eqnarray}\label{qw1}
\int_{R_1}^{R_2} u_r^2 r^{n-1}\, dr=\int_{R_1}^{R_2}a(r)u^p r^{n-1}\, dr.
\end{eqnarray}
It also from the definition of $\lambda_1$,  the best constant in Hardy inequality, that 
\begin{equation}\label{hardyy}\lambda_1 \int_{R_1}^{R_2} \frac{u^2}{r^2}r^{N-1}\, dr\leq \int_{R_1}^{R_2} u_r^2 r^{N-1} \, dr. \end{equation}
It follows from (\ref{qw1}) by writing $M(u,v)$ in polar coordinates that  
\begin{eqnarray*}
M(u,v)&=&\int_{R_1}^{R_2} \int_0^{\frac{\pi}{2}}\Big(\psi^2u_r^2+\frac{u^2\psi'^2}{r^2}-(p-1)a(r)u^{p}\psi^2\Big)r^{N-1}\omega(\theta) \,d\theta\, dr\\
 &=& \int_{R_1}^{R_2} \int_0^{\frac{\pi}{2}}\frac{u^2\psi'^2}{r^2}r^{N-1}\omega(\theta)  \,d\theta\, dr-(p-2)\int_{R_1}^{R_2} \int_0^{\frac{\pi}{2}}\psi^2u_r^2 r^{N-1}\omega(\theta) \,d\theta\, dr,
 \end{eqnarray*}
 where $\omega(\theta)=\cos^{m-1}(\theta) \sin^{n-1}(\theta)$.
 This together with the definition of $\mu_1=2N$ in (\ref{var_mu}) and the inequality  (\ref{hardyy}) imply that 
\begin{eqnarray*}
M(u,v)
 &=& 2N\int_{R_1}^{R_2} \int_0^{\frac{\pi}{2}}\frac{u^2\psi^2}{r^2}r^{N-1}\omega(\theta) \,d\theta\, dr-(p-2)\int_{R_1}^{R_2} \int_0^{\frac{\pi}{2}}\psi^2u_r^2 r^{N-1}\omega(\theta) \,d\theta\, dr\\
 &=& \int_0^{\frac{\pi}{2}} |\psi(\theta)|^2\omega(\theta) \,d\theta \Big (   2N\int_{R_1}^{R_2} \frac{u^2}{r^2}r^{N-1}\, dr-(p-2)\int_{R_1}^{R_2} u_r^2 r^{N-1}\, dr\Big)\\
 &\leq & \int_0^{\frac{\pi}{2}} |\psi(\theta)|^2\omega(\theta) \,d\theta  \Big (  \frac{2N}{\lambda_1}\int_{R_1}^{R_2} u^2_r r^{N-1}\, dr-(p-2)\int_{R_1}^{R_2} u_r^2 r^{N-1} \, dr \Big)\\
 &=&\int_0^{\frac{\pi}{2}} |\psi(\theta)|^2\omega(\theta) \,d\theta \int_{R_1}^{R_2} u_r^2 r^{N-1} \, dr\Big(\frac{2N}{\lambda_1}-(p-2)\Big)<0.
\end{eqnarray*}

 Set $\gamma_\sigma (\tau)= \tau({u}+ \sigma v)l$, where $l>0$ is chosen in such a way that $E\big( ({u}+ \sigma v)l\big)\leq 0$ for all $|\sigma|\leq 1$. Note that $\gamma_{\sigma}\in \Gamma$. We shall show that there exists $\sigma>0$ such that for every $\tau\in [0,1]$ one has  $E(\gamma_\sigma(\tau))< E({u})$, and  therefore, 
\[c\leq \max_{\tau\in[0,1]} E(\gamma_\sigma(\tau))< E({u}),\]
which leads to a contradiction since $E(u)=c.$
Note first that there exists a unique  smooth real  function $g$ on a small neighbourhood of zero with $g'(0)=0$ and $g(0)=1/l$ such that 
$\max_{\tau\in[0,1]} E(\gamma_\sigma(\tau))=E\big (g(\sigma)({u}+ \sigma v)l\big).$
We now define $h: \R \to\R$ by $$h(\sigma)=E\big (g(\sigma)({u}+ \sigma v)l\big)-E(u).$$ Clearly we have $h(0)=0$.  Note also that $h'(0)=0$ due to the facts that $E'(u)=0$ and $\int \psi \omega(\theta) \, d\theta=0. $ Finally 
 $h''(0)<0$ due to (\ref{qw}).  This in fact show that 
 \[\max_{\tau\in[0,1]} E(\gamma_\sigma(\tau))=E\big (g(\sigma)({u}+ \sigma v)l\big)<E(u),\]
for small $\sigma>0$  as desired.
\hfill $\Box$ \\

Recall from (\ref{cK}) that 
\begin{equation*}K=K(m,n)=\left\{ 0 \le u \in H_{0,G}^1(\Omega):  su_t-tu_s \le 0 \mbox{ a.e. in  } \widehat{\Omega} \right\}, 
\end{equation*}
which corresponds to the decomposition $\R^{m}\times \R^{n}$ of the annulus $\Omega=\{x:\,  R_1< |x|<R_2\}$  in 
$ \R^N$ with $N=m+n.$  We have the following result regarding the distiction of solutions for different decompositions of $ \R^N$.

\begin{lemma}\label{mm'}
Let $1\leq n <n'\leq \lfloor \frac{N}{2}\rfloor$ and set $m=N-n$, $m'=N-n'.$   Let $u_{m,n} \in K(m,n)$ and $u_{m',n'}\in K(m',n')$ be the  solutions obtained in Theorem \ref{main} corresponding to the decomposition $\R^{m}\times \R^{n}$ and  $\R^{m'}\times \R^{n'}$ of $\R^N$ respectively. Then $u_{m,n} \not = u_{m',n'}$ unless they are both radial functions. 
\end{lemma}
\begin{proof}
Let $u_{m,n} =u_{m',n'}=u.$  We shall show that $u$ must be radial.  It follows that $u(x)=f(s,t)=g(s',t')$ for two  functions $f$ and $g$ where 
\[ s^2:= { x_1^2 + \cdot \cdot \cdot  + x_{m}^2}, \qquad t^2:={x_{m+1}^2 + \cdot \cdot \cdot + x_N^2 },  \]
and
\[ {s'}^2:= { x_1^2 + \cdot \cdot \cdot  + x_{m'}^2}, \qquad {t'}^2:={x_{m'+1}^2 + \cdot \cdot \cdot + x_N^2 }.  \]

By assuming  $x_i=0$  for $i\not = x_1, x_{m'}$  we obtain that
\[f(|x_1|, |x_{m'}|)=g(\sqrt{x_1^2+x_{m'}^2}, 0),\]
from which we obtain that $f$ must be a radial function.  Similarly, one can show that $g$ is a radial function.
\end{proof}

\noindent
\textbf{Proof of Theorem \ref{multiplicity_100}.}
  It follows from Theorem \ref{main}   that for each $n\leq k$ and $m=N-n$ equation (\ref{eqz}) has a solution of the form $u_{m,n}=u_{m,n}(s,t)$ where
\[ s^2:= { x_1^2 + \cdot \cdot \cdot  + x_{m}^2}, \qquad t^2:={x_{m+1}^2 + \cdot \cdot \cdot + x_N^2 }, \]
provided $p< \frac{2n+2}{n-1}.$  Since $n\leq k$ we must have that $$\frac{2 k+2}{k-1}\leq \frac{2n+2}{n-1}.$$
Also by Theorem \ref{nnr} the solution $u_{m,n}$ is non-radial provided  \[p-2>\frac{2N}{ \lambda_1}.\]
Thus,  for each $n \in \{1,...,k\}$ we have a non-radial solution $u_{m,n}.$
On the other hand by Lemma \ref{mm'} we have that Then $u_{m,n} \not = u_{m',n'}$ for all $n\not =n'.$  This indeed implies that we have $k$ positive non-radial solutions.
\hfill $\Box$ \\

\noindent
\textbf{Proof of  Corollary \ref{annul_resu}.} 

\noindent
1. By assuming $k=1$ in Theorem \ref{multiplicity_100}, provided $ \lambda_1>0$ we obtain 1,  but this follows from the classical Hardy inequality. \\

\noindent
2.   By assuming $k=\Bigl\lfloor\frac{N}{2} \Bigr\rfloor$ in Theorem \ref{multiplicity_100} we obtain that there are $\Bigl\lfloor\frac{N}{2} \Bigr\rfloor$ positive non-radial solutions provided that 
$$2+\frac{2N}{\lambda_1}<p<\frac{2\Bigl\lfloor\frac{N}{2} \Bigr\rfloor+2}{\Bigl\lfloor\frac{N}{2} \Bigr\rfloor-1}.$$
We shall now show that $\lambda_1$ can be sufficiently large under 2-a and 2-b and therefore we have $\Bigl\lfloor\frac{N}{2} \Bigr\rfloor$ positive non-radial solutions provided 
$$2<p<\frac{2\Bigl\lfloor\frac{N}{2} \Bigr\rfloor+2}{\Bigl\lfloor\frac{N}{2} \Bigr\rfloor-1}.$$
2-a. Set $ \Omega_R:=\{ x \in \IR^N: R<|x|<R+1\}$.   Set \[\lambda_R=\inf_{0 \neq w \in H^1_0(\Omega_R)}\frac{\int_{\Omega_R} |\nabla w|^2 \,dx}{\int_{\Omega_R} \frac{| w|^2 }{|x|^2}\,dx}=: \inf_{0 \neq w \in H_0^1(\Omega_R)} H(w),\] where $\Omega_R:=\{x \in \IR^N:  R<|x|<R+1 \}$.   We will show that $ \frac{\lambda_R}{R^2} \rightarrow \pi^2$ as $ R \rightarrow \infty$ which then shows $ p-2> \frac{2N}{\lambda_R}$ for $R$ large.  Note there is no loss of compactness and so $ \lambda_R$ is attained at some $ w_R$ and note that we can assume $ w_R>0$ and that $ w_R$ is radial, which we explain in a moment.  Note that $ w_R>0$ solves 
\begin{equation} \label{eigen_R}
\left\{\begin{array}{ll}
-\Delta w_R(x) = \frac{\lambda_R w_R(x)}{|x|^2} &  \mbox{ in } \Omega_R, \\
w_R= 0 &   \mbox{ on } \pOm_R.
\end {array}\right.
\end{equation} So note that $\lambda_R$ is the first eigenvalue of this weighted problem and hence if $w_R>0$ was not radial we could use the rotational invariance of the equation to see the dimension of the first eigenspace was larger than one which would be a contradiction and hence $ w_R$ is radial.   Set $ v_R(r)=w_R(R+r)$ and hence 
\begin{equation} \label{eigen_v}
-v_R''(r)-\frac{(N-1)}{R+r} v_R'(r)= \frac{\lambda_R}{(R+r)^2} v_R(r) \mbox{ in } 0<r<R, \qquad v_R(0)=v_R(1)=0.
\end{equation} After normalizing we can assume $ \max_{[0,1]} v_R=1$ and one can show that $ v_R'(r)$ is bounded in $ R$.  We now want to show that $ \lambda_R$ grows at most like $ R^2$.  To see this fix some $ \psi \in C_c^\infty(B_\frac{1}{4})$ and let $ |x_R|=R+\frac{1}{2}$ then note that $ \lambda_R \le H( \psi( \cdot - x_R))$.  Writing out the details gives the desired bound.  Now let $ R_m \rightarrow \infty$ and set $ v_m(r)=v_{R_m}(r)$.  By passing to a subsequence we can assume that $ \frac{\lambda_{R_m}}{R_m^2} \rightarrow \tilde{\lambda}$. We can pass to the limit in (\ref{eigen_v}) to see there is some $v \ge 0$ such that $v_m \rightarrow v$ in $C^{0,\delta}[0,1]$ and hence $v$ satisfies $-v''(r)= \tilde{\lambda}v(r)$ in $ (0,1)$ with $ v(0)=v(1)=0$ and $ \sup_{(0,1)}v=1$.  Then we must have $ v(r)= \sin( \pi r)$ and $ \tilde{\lambda}= \pi^2$ and hence $ \frac{\lambda_{R_m}}{R_m^2} \rightarrow \pi^2$. \\

\noindent
2-b.  Let $\Omega_R:=\{x \in \IR^N: R<|x|<\gamma(R) \}$ and let $ w_R=w_R(r)>0$ be the minimizer and hence $ w_R$ satisfies
\[ -\Delta w_R(r) = \frac{ \lambda_R w_R(r)}{r^2} \quad  \mbox{ in } R<r<\gamma(R),\] with $ w_R(R)=w_R(\gamma(R))=0$.  Set $v_R(r)=w_R( R+(\gamma(R) -R)r)$ for $ r \in (0,1)$ and note $ v_R(0)=v_R(1)=0$.  Then 
\begin{equation} \label{exp_ann}
-v_R''(r) - (N-1) C_R(r) v_R'(r)= D_R(r) v_R(r) \quad \mbox{ in } 0<r<1,
\end{equation} where 
\[ C_R(r):= \frac{\gamma(R)-R}{R+(\gamma(R)-R) r}, \; \; D_R(r)= \frac{ \lambda_R ( \gamma(R)-R)^2}{( R+ (\gamma(R)-R)r)^2}.\]  Note that $ 0 \le C_R(r) \le C_R(0) \rightarrow 0$ as $ R \rightarrow \infty$ after considering the assumptions on $ \gamma(R)$.  Also we have
\[ 0 \le D_R(r) \le D_R(0)= \frac{ \lambda_R ( \gamma(R)-R)^2}{R^2}.\]  Let $R_m \rightarrow \infty$ and by passing to a subsequence we can assume that either $ D_{R_m}(0) \rightarrow \infty$ or its bounded.  In the first case we have that $ \lambda_{R_m} \rightarrow \infty$ which is what we hoped to prove.  So now we can assume that $D_{R_m}(0)$ is bounded and that we can write $ D_R(r)= D_R(0) T_R(r)$ where $T_R \rightarrow 1$ uniformly on $[0,1]$.  By normalizing $ v_{m}=v_{R_m}$ we can assume $ \max_{[0,1]}v_m=1$ and we can pass to the limit in (\ref{exp_ann}) to see that $ v_m \rightarrow v$ in $ C^{0,\delta}[0,1]$ and $ D_{R_m}(r) \rightarrow \tilde{\lambda}$ uniformly in $[0,1]$ and so $ v$ solves $ -v''(r) = \tilde{\lambda} v(r)$ in $ 0<r<1$ with $ v(0)=v(1)=0$ and $ v \ge 0$ with $ \max_{[0,1]}v=1$ and hence $ v(r)=\sin( \pi r)$ and $ \tilde{\lambda}= \pi^2$. In any case from this we see that $ \lambda_R \rightarrow \infty$ and again we are done. 
\hfill $\Box$

\end{document}